\documentclass[a4paper,11pt,reqno]{amsart}

\usepackage[T1]{fontenc}
\usepackage{lmodern}
\usepackage{amsmath,amssymb,amsthm,mathtools}
\usepackage[top=3cm,bottom=3cm,left=3.15cm,right=3.15cm]{geometry}
\usepackage[colorlinks=true,linkcolor=blue,citecolor=blue,urlcolor=blue]{hyperref}
\usepackage[capitalise,nameinlink,noabbrev]{cleveref}
\usepackage{microtype}
\usepackage{booktabs,array}
\usepackage{tikz}
\usepackage{tikz-cd}
\usepackage{needspace}

\AddToHook{env/proof/end}{\nopagebreak[4]}

\newtheorem{theorem}{Theorem}[section]
\newtheorem{proposition}[theorem]{Proposition}
\newtheorem{lemma}[theorem]{Lemma}
\newtheorem{corollary}[theorem]{Corollary}
\theoremstyle{definition}

\newtheoremstyle{spacedremark}
  {9pt plus 2pt minus 1pt}{9pt plus 2pt minus 1pt}
  {\normalfont}{}{\bfseries}{.}{0.5em}{}
\theoremstyle{spacedremark}
\newtheorem{remark}[theorem]{Remark}
\AddToHook{env/remark/begin}{%
  \pushQED{\qed}%
}
\AddToHook{env/remark/end}{\popQED}

\newcommand{\C}{\mathbb C}
\newcommand{\R}{\mathbb R}
\newcommand{\Z}{\mathbb Z}
\newcommand{\T}{\mathbb T}
\newcommand{\Pone}{\mathbb P^1}

\newcommand{\Id}{\operatorname{Id}}

\title[The genus-one case of Calabi's problem]
{A Complex Structure on \(\T^2\times S^4\):\\
The Genus-One Case of Calabi's Problem}

\author[M. Gao]{Min Gao}
\address{Chern Institute of Mathematics,
Nankai University, Tianjin 300071, P. R. China}
\email{mingao@mail.nankai.edu.cn}

\author[Z. Z. Tang]{Zizhou Tang}
\address{Chern Institute of Mathematics and LPMC,
Nankai University, Tianjin 300071, P. R. China}
\email{zztang@nankai.edu.cn}

\author[W. J. Yan]{Wenjiao Yan}
\address{School of Mathematical Sciences, Laboratory of Mathematics and Complex Systems,
Beijing Normal University, Beijing 100875, P. R. China}
\email{wjyan@bnu.edu.cn}

\date{}

\thanks{This work was partially supported by the NSFC
(Nos.~12271038, 12371048 and 12526205) and Nankai
Zhide Foundation.}

\hypersetup{
  pdftitle={A Complex Structure on T2 x S4: The Genus-One Case of Calabi's Problem},
  pdfauthor={Min Gao, Zizhou Tang, and Wenjiao Yan},
  pdfsubject={A complex structure on the smooth manifold T2 x S4},
  pdfkeywords={Calabi's problem, complex structures, torus degeneration, smooth topology}
}

\begin{document}

\begin{abstract}
We construct a complex structure on \(\T^2\times S^4\), giving an affirmative answer to
the genus-one case of Calabi's problem.  The resulting compact complex
threefold is non-K\"ahler and admits no biholomorphic product
decomposition.
\end{abstract}

\maketitle

\section{Introduction}

The sphere \(S^4\) admits no complex structure; in fact,
it does not even admit an almost complex structure. Calabi \cite{Calabi1958} asked whether the product
\(V^2\times S^4\), where \(V^2\) is a closed orientable surface,
admits a complex structure. In the genus-zero case, a construction
on \(S^2\times S^4\) has recently been proposed in
\cite{LiangLiuZhang2026}, conditional on the validity of the proposed
\(S^6\) construction in \cite{SixSphere2026}.
In this paper, we construct a complex structure on
\(\T^2\times S^4\), giving an affirmative answer to Calabi's
problem in genus one. Our construction relies neither on a complex structure on
\(S^6\) nor on results from \cite{SixSphere2026}. 
Throughout this paper, \(\T^r=\R^r/\Z^r\).

More generally, little is known about the existence of complex
structures on \(S^4\times M\), where \(M\) is a closed smooth manifold.
The existence of a complex structure depends, in general,
on the underlying smooth structure. Whether \(S^4\) has a unique smooth structure up to diffeomorphism
remains unknown \cite{HassKirby2025}. Ge-Tang
\cite{GeTang2013} showed that a smooth homotopy four-sphere
\(\Sigma^4\) admitting a properly isoparametric function for some
Riemannian metric is diffeomorphic to \(S^4\).
The constructions of Meersseman
\cite
{Meersseman2000} and Bosio-Meersseman \cite
{BosioMeersseman2006}
yield complex structures on manifolds such as
\(\bigl(\#^{5}(S^3\times S^4)\bigr)\times S^1\), where
\(\#^5\) denotes the connected sum of five copies.
In this example, \(S^4\) is a factor of the connected-sum
summands, not of the whole manifold. A related question arises
from the isoparametric approach of Tang-Yan
\cite{TangYan2022} and Qian-Tang-Yan \cite{QianTangYan2026}. In the definite OT--FKM case with first
multiplicity \(m=4\), a complex structure on the focal submanifold \(M_+\)
was constructed in \cite{QianTangYan2026}. 
The associated isoparametric hypersurface is diffeomorphic to
\(S^4\times M_+\), but the existence of a complex structure
on this product was left open.
This paper gives the first
construction of a complex structure on \(\T^2\times S^4\), where $S^4$ is a product factor.

We construct the regular part of the fibration as a holomorphic
family of complex two-tori over the punctured complex plane \(\C^*\), with periods specified
explicitly. Each fibre is diffeomorphic to \(\T^4\).
Holomorphic torus families also occur in earlier constructions
on smooth products of complex curves with \(\T^4\).
Blanchard constructed non-K\"ahler
complex structures on \(S^2\times\T^4\), with a holomorphic
projection to the complex projective line \(\Pone\) and complex two-torus fibres
\cite
{Blanchard1956}. Calabi obtained examples
on \(M\times\T^4\), where \(M\) is a hyperelliptic curve of odd
genus at least three \cite
{Calabi1958}. Sommese treated
the construction over \(\Pone\) in the broader setting of
quaternionic manifolds
\cite
{Sommese1975}.
For a general formulation over compact complex curves, see
\cite
{Catanese2004}. In these earlier torus examples, the
underlying smooth fibration is a product, but the complex
structures of the fibres vary. This variation is described by
periods: lattice generators in \(\C^2\) that depend locally
holomorphically on the base parameter.

Identify the projective line
\(\Pone\) with \(\C\cup\{\infty\}\), where \(0\) is the
origin of \(\C\) and \(\infty\) is the point at infinity.
Our torus family is defined over
\(\Pone\setminus\{0,\infty\}=\C^*\) and degenerates at the
two omitted points. We complete the family across these points
with a smooth complex total space and identify the resulting
compact manifold with \(\T^2\times S^4\).
The period matrices, analytic fillings, and holomorphic gluing
maps are all given explicitly.

\Needspace{12\baselineskip}
The main result of this paper is the following.

\begin{theorem}\label{thm:main}
\(\T^2\times S^4\) admits a complex structure. More precisely, there exist a compact complex threefold \(X\), a diffeomorphism
\[
 X\cong_{\mathrm{diff}}\T^2\times S^4,
\]
and a surjective holomorphic map \(g:X\to\Pone\) such that:
\begin{enumerate}
\item every regular fibre is a complex two-torus;
\item the critical locus is the disjoint union of two complex tori
of complex dimension \(1\), lying over \(0\) and \(\infty\),
respectively.
\end{enumerate}
\end{theorem}

The resulting threefold is non-K\"ahler and admits no nontrivial
biholomorphic product decomposition. In fact, these properties
hold for every complex structure on the same smooth manifold.
\begin{proposition}\label[proposition]{prop:nonproduct}
Let \(Z\) be a compact complex manifold diffeomorphic to
\(\T^2\times S^4\). Then \(Z\) is non-K\"ahler and is not
biholomorphic to a product of compact complex
manifolds.
\end{proposition}

We outline the proof of \cref{thm:main}. A complex two-torus
is a quotient \(\C^2/\Lambda\), where the lattice \(\Lambda\)
is generated by four real-linearly independent vectors.
On the logarithmic cover \(q=e^{2\pi it}\), these generators
are chosen as the columns of a holomorphic matrix \(P(t)\).
We verify that the associated real matrix \(L_{P(t)}\) is
nonsingular and that \(P(t+1)=P(t)M\), where \(M\) is the
integral monodromy matrix fixed in \cref{sec:topology}.
The lattice quotients then descend to a holomorphic torus family
over \(\C^*\).

To obtain a compact total space, we add singular fibres over
\(0\) and \(\infty\). Near \(0\), the period transformations
extend to an infinite chain of charts \(xy=q\).
The extended action is free and proper, and the
quotient has a proper projection to a disc. The
same construction applies at \(\infty\). These are the local charts
of the rank-one Tate model
\cite[Section~1.1.3]{KajiwaraKatoNakayama2008}; related Tate
completions and quadratic logarithmic translations occur in
\cite[Section~3]{DiScalaKasuyaZuddas2017}.
A holomorphic translation identifies the two completed ends
over their common regular part.

The next task is to identify the underlying smooth manifold, we deform each
completed compact block to a product model, keeping the base
and all boundary period coordinates fixed. After absorbing a
fixed reflection into the torus coordinates on one block, the
attaching map is the identity on the real two-torus factor. Thus
\(X\cong_{\mathrm{diff}}\T^2\times Y\), where \(Y\) is a
closed smooth four-manifold obtained by gluing two one-node blocks.
The gluing calculation gives \(\pi_1(Y)=0\) and
\(\chi(Y)=2\), so \(Y\) is a smooth homotopy four-sphere.
Hirsch's theorem \cite[Theorem~6(a)]{Hirsch1965} then yields
\(\T^2\times Y\cong_{\mathrm{diff}}\T^2\times S^4\),
without requiring an identification of \(Y\) itself with the
standard smooth \(S^4\). The product identification can moreover
be chosen so that a real two-torus in a regular fibre of \(g\)
represents a generator of \(H_2(X;\Z)\);
see \cref{rem:fibre-generator}.



Taking the product of \(k\) copies of the threefold in
\cref{thm:main}, together with a complex torus of complex
dimension \(m-k\) when \(m>k\), gives the following corollary.
\begin{corollary}\label[corollary]{cor:products-four-spheres}
Let \(m\ge k\ge1\) be integers. Then
\(\T^{2m}\times(S^4)^k\) admits a complex structure of complex
dimension \(m+2k\).
\end{corollary}

The paper is organized as follows. We fix the boundary markings
and describe the one-node blocks in \cref{sec:topology}.
\Cref{sec:periods,sec:filling} construct the torus family and
its analytic completions. In \cref{sec:relative}, we establish
the relative smooth product identifications needed for gluing.
The holomorphic gluing is carried out in \cref{sec:global},
and the smooth topology is determined in \cref{sec:recognition},
completing the proof of \cref{thm:main}.
Finally, \cref{sec:consequences} proves \cref{prop:nonproduct}.

\section{Markings and one-node topology}\label{sec:topology}
This section fixes the notation needed for the construction and its
topological identification. A marking is a choice of an ordered
integral basis for the first homology of a torus fibre. In
\cref{subsec:torus-markings}, we specify the monodromy and the change
of marking at \(\infty\). These enter the period construction in
\cref{sec:periods}. The description of the
Tate one-node block in \cref{subsec:node-topology} includes a
computation of the homomorphism induced by its boundary inclusion.
Together with the attaching map in
\cref{sec:recognition}, this computation will determine the
fundamental group of the four-dimensional factor of \(X\).

\subsection{Lattice markings and monodromy}\label{subsec:torus-markings}
Let \(F\cong\T^2\) be a regular fibre of the positive one-node
model described in \cref{subsec:node-topology}. Choose an ordered
integral basis \((m,\ell)\) of \(H_1(F;\Z)\), where \(m\) is the
primitive vanishing class and \(m\cdot\ell=1\).
Throughout this subsection, matrices act on column vectors in
the indicated ordered basis. The monodromy around a
counterclockwise loop about the critical value is represented by
\begin{equation}\label{eq:T}
 T=\begin{pmatrix}1&1\\0&1\end{pmatrix},
 \qquad Tm=m,\quad T\ell=\ell+m.
\end{equation}
The geometric choice of \(m,\ell\)
and the sign of this monodromy are explained in
\cref{subsec:node-topology}; the same matrix follows directly
from \eqref{eq:tate-coordinate-seam}.

In the smooth product models used later, a second real two-torus
does not degenerate. Write \(a_T,b_T\) for an ordered integral
basis of its first homology. The first integral homology of the
product fibre is therefore
\[
 \Lambda=\Z a_T\oplus\Z m\oplus\Z b_T\oplus\Z\ell
 \cong\Z^4.
\]
We denote the following ordered integral basis of \(\Lambda\)
by \(\mathcal B\):
\begin{equation*}
 \mathcal B=(a_T,m,b_T,\ell)=(e_1,e_2,e_3,e_4).
\end{equation*}
Here the marking identifies these four homology classes with
the standard coordinate vectors \(e_1,e_2,e_3,e_4\) of \(\Z^4\).
The local monodromy fixes \(a_T,b_T\) and acts by \(T\) on
\((m,\ell)\). In the ordered basis \(\mathcal B\), grouped as
\((a_T,m)\mid(b_T,\ell)\) for the period construction, its
matrix \(M\in\operatorname{SL}(4,\Z)\) is
\begin{equation*}
 B=\begin{pmatrix}0&0\\0&1\end{pmatrix},
 \qquad M=\begin{pmatrix}I_2&B\\0&I_2\end{pmatrix}.
\end{equation*}
The matrix \(B\) is not itself a monodromy automorphism.
The geometric splitting of the lattice is instead
\[
 \Lambda=R\oplus D,\qquad
 R=\Z e_1\oplus\Z e_3,\quad D=\Z e_2\oplus\Z e_4.
\]
Thus \(R\) contains the two nondegenerating directions,
\(D\) contains the two directions of the degenerating torus,
and \(M|_R=\Id\), \(M|_D=T\).
This induces a splitting of the underlying real tori; it need not
be a holomorphic splitting of the complex fibres.

The matrix \(M\) will be the prescribed monodromy of the torus
family constructed in \cref{sec:periods}.

At \(\infty\) we use the integral re-marking
\begin{equation}\label{eq:G0}
 G_0=\operatorname{diag}(1,1,-1,-1)=R_T\oplus R_V,
\end{equation}
where the direct sum refers to \(R\oplus D\):
\[
 R_T(a_T)=a_T,\quad R_T(b_T)=-b_T,\qquad
 R_V(m)=m,\quad R_V(\ell)=-\ell.
\]
The two reflections have determinant \(-1\), while \(G_0\)
has determinant \(1\). Direct calculation gives
\begin{equation*}
 \begin{aligned}
 G_0MG_0^{-1}
 =\begin{pmatrix}I_2&-B\\0&I_2\end{pmatrix}=M^{-1},\qquad
 R_VTR_V^{-1}
 =\begin{pmatrix}1&-1\\0&1\end{pmatrix}=T^{-1}.
 \end{aligned}
\end{equation*}
Here \(q\) is the affine coordinate on \(\Pone\), and
\(w=q^{-1}\) is the coordinate at \(\infty\). The change of marking
\(G_0\) compensates for the reversal of a loop under
\(q\mapsto q^{-1}\). Thus the local monodromy at each end is
represented by \(M\) in its own marking. We shall use a positive
Tate model at each end.

\subsection{Tate one-node block}\label{subsec:node-topology}
We now describe the four-dimensional block that occurs in the smooth
product models of \cref{sec:relative}. Its real fibre coordinates
will also be used to write the attaching map in
\cref{sec:recognition}.

By a \emph{node} of a complex curve we mean a singular point whose
analytic germ is isomorphic to the germ of
\(\{(x,y)\in\C^2:xy=0\}\) at the origin; its two smooth local
branches meet transversely. The local smoothing is \(xy=q\):
the total space is smooth, while the central fibre has a node.

For \(\kappa>0\) and a sufficiently small closed disc
\(\Delta\) centred at \(0\), let
\(\pi_\kappa:N_\kappa\to\Delta\) denote the standard one-node
completion of
\(\C^*/(\kappa q)^\Z\) over \(q\ne0\).
This is the classical Tate nodal completion; see
\cite{Kodaira1963} and
\cite[Section~3]{DiScalaKasuyaZuddas2017}.
The completion is a proper holomorphic map from a smooth complex
surface over an open neighbourhood of \(\Delta\);
\(N_\kappa\) is its restriction over the closed disc.
Thus \(N_\kappa\) is a compact smooth four-manifold with boundary.
Its regular fibres \(F_q=\pi_\kappa^{-1}(q)\) are smooth
genus-one curves. The central fibre \(F_0\) is an irreducible
rational curve with one node, and the unique critical point
has the local model \(q\circ\pi_\kappa(x,y)=xy\).
This is the positive nodal model with respect to the complex
orientations. A smaller disc and
the addition of a regular outer collar give the same type of
marked piece. We use \(N_0,N_\infty\) for two copies, with their
own local base coordinates.

The basis used in \cref{subsec:torus-markings} can be chosen as follows.
For a sufficiently small positive real \(q_0\), the circle
\[
 \gamma_{q_0}(s)
 =\bigl(\sqrt{q_0}\,e^{2\pi is},
        \sqrt{q_0}\,e^{-2\pi is}\bigr),
 \qquad s\in\R/\Z,
\]
lies in \(F_{q_0}\) and contracts to the node as \(q_0\to0\).
It is the \emph{vanishing cycle}. Since \(F_0\) is irreducible,
this circle is nonseparating. Choose an oriented simple closed
curve \(\delta\subset F_{q_0}\) meeting \(\gamma_{q_0}\)
transversely once with positive intersection.
Then \(m=[\gamma_{q_0}]\), \(\ell=[\delta]\) form an ordered
integral basis of \(H_1(F_{q_0};\Z)\), with \(m\cdot\ell=1\).
In particular, \(m\) is primitive: it is not a multiple \(kv\)
with \(|k|>1\) and \(v\in H_1(F_{q_0};\Z)\).
The monodromy along \(q(s)=q_0e^{2\pi is}\), \(0\le s\le1\),
is a right-handed Dehn twist about this vanishing cycle
\cite[Section~2]{DiScalaKasuyaZuddas2017}.
It fixes \(m\) and sends \(\ell\) to \(\ell+m\), as recorded
in \eqref{eq:T}.

For the two blocks used below, \(\kappa=e^{-2\pi}\).
We introduce real Tate coordinates on their regular parts by writing
\[
 q=re^{2\pi i\theta},
 \quad
 t=\theta+\frac{\log r}{2\pi i},
 \quad
 \theta\in\R,\quad 0<r<\kappa^{-1}.
\]
The relation \(e^{2\pi i(t+i)}=\kappa q\) allows us to
pass between additive and multiplicative coordinates:
\[
 \begin{aligned}
 \C/(\Z+(t+i)\Z)
 &\xrightarrow{\;\cong\;}\C^*/(\kappa q)^\Z,\\
 [\zeta]&\longmapsto[e^{2\pi i\zeta}].
 \end{aligned}
\]
Since
\(
 \operatorname{Im}(t+i)
 =1-\frac{\log r}{2\pi}>0,
\)
the numbers \(1\) and \(t+i\) always form a real basis of \(\C\).
Hence each point of the fibre has a unique expression
\([\mu+(t+i)\nu]\), with \(\mu,\nu\in\R/\Z\).
Because 
\(
 \mu+(t+1+i)\nu=(\mu+\nu)+(t+i)\nu,
\)
the corresponding fibre coordinates are related by
\begin{equation}\label{eq:tate-coordinate-seam}
 (\theta+1,\mu,\nu)
 \sim
 (\theta,\mu+\nu,\nu).
\end{equation}
We use \([\theta,\mu,\nu]\) to denote the equivalence class under
this identification. Thus
\begin{equation}\label{eq:real-tate-coordinates}
 \begin{aligned}
 (r,[\theta,\mu,\nu])
 &\longleftrightarrow
 \bigl(q,[\mu+(t+i)\nu]_{\Z+(t+i)\Z}\bigr)\\
 &\longleftrightarrow
 \bigl(q,[e^{2\pi i(\mu+(t+i)\nu)}]_{(\kappa q)^\Z}\bigr).
 \end{aligned}
\end{equation}
We call \((r,[\theta,\mu,\nu])\) \textit{the real Tate coordinates}.

When \(r=r_0\) is fixed, these coordinates give the
mapping torus
\[
 \frac{\R_\theta\times(\R/\Z)^2_{\mu,\nu}}
 {(\theta+1,\mu,\nu)\sim(\theta,\mu+\nu,\nu)}.
\]
The marked loops based at \([0,0,0]\) are
\[
 \begin{aligned}
 m(a)&=[0,a,0],\\
 \ell(a)&=[0,0,a],\\
 \theta(a)&=[a,0,0],
 \end{aligned}
 \qquad 0\le a\le1.
\]

For the blocks over the unit discs used later, the
boundary is given by \(r=1\). Allowing \(r\) to vary
near \(1\) gives the regular collar used in the gluing. These are smooth real coordinates on the family over
the punctured disc. At the central fibre, we use the
plumbing charts.

\begin{lemma}\label[lemma]{lem:node-topology}
A one-node block \(N\) with primitive vanishing meridian
\(m\) satisfies
\[
 \pi_1(N)=\Z\langle\ell\rangle,
 \quad \chi(N)=1.
\]
For the boundary loops defined above, the inclusion
\(i:\partial N\hookrightarrow N\) induces
\[
 i_*(m)=i_*(\theta)=1,
 \quad i_*(\ell)=\ell.
\]
\end{lemma}

\begin{proof}
Let \(C_0\) be the central nodal fibre.
For a sufficiently small base disc, the total space
deformation retracts onto \(C_0\)
\cite[Section~4.1]{Deligne1974}.
Adding a regular outer collar does not change the
homotopy type.

To identify the boundary homomorphism, let \(F\subset\partial N\)
be the fibre over the positive real boundary point. Transport
\(F\) inward along the positive real radius, keeping the real
Tate coordinates \((\mu,\nu)\) fixed on the regular outer collar.
In the plumbing model \(xy=q\), letting \(q>0\) decrease to
\(0\) collapses the core circle of the smoothing annulus,
which represents the vanishing class \(m\). Together with a
smooth trivialization of its complement along this radial
segment, this gives a collapse map \(c:F\to C_0\).
Regarded as maps into \(N\), this map and the inclusion of
\(F\) are homotopic, with the basepoint following the unit
section. Using this path to identify basepoints, the inclusion
\(F\hookrightarrow N\), followed by the retraction to \(C_0\),
therefore induces \(c_*\) on fundamental groups.

Since \(C_0\) is obtained by collapsing a representative of
\(m\) in \(F\),
\[
 N\simeq C_0\cong\T^2/m\simeq S^2\vee S^1.
\]
The map \(c\) kills the class \(m\) and sends \(\ell\)
to a generator of the circle factor.
Thus
\[
 \begin{aligned}
 \pi_1(N)
 &\cong\langle m,\ell\mid[m,\ell]=1,\ m=1\rangle
   =\Z\langle\ell\rangle,\\
 \chi(N)&=\chi(S^2)+\chi(S^1)-1=1.
 \end{aligned}
\]
In the standard Tate completion, the unit section represented
by \(1\in\C^*/(\kappa q)^\Z\) extends to the smooth point
represented by \(1\) in the normalization \(\Pone\) of
\(C_0\). On the boundary, this section is given by
\(\mu=\nu=0\), so its boundary is exactly the marked loop
\(\theta\). Thus \(\theta\) bounds the section disc in \(N\).
Hence \(i_*(m)=i_*(\theta)=1\) and \(i_*(\ell)=\ell\).

\end{proof}

\section{Construction of the holomorphic torus family}\label{sec:periods}

In this section, we construct the regular part of the map in
\cref{thm:main} as a holomorphic family
\(p:\mathcal G\to\C^*\) of complex two-tori with the monodromy
\(M\) fixed in \cref{sec:topology}. A real determinant calculation
verifies that the chosen periods generate full lattices.
Quotienting by lattice translations in the fibres and then by
the deck action on the logarithmic cover gives the required
family. Its completion at \(0\) and \(\infty\) is treated in
\cref{sec:filling}.

By a \emph{holomorphic family of compact complex manifolds} we mean
a proper holomorphic submersion
\(\pi:\mathcal X\to B\) between complex manifolds. A \emph{holomorphic isomorphism} between two families \(\pi:\mathcal X\to B\) and \(\pi':\mathcal X'\to B'\)
over a biholomorphism \(\beta:B\to B'\) is a biholomorphism
\(F:\mathcal X\to\mathcal X'\) satisfying
\(\pi'\circ F=\beta\circ\pi\).

\subsection{Periods and real linear independence}
\label{subsec:linear-periods}

For \(P\in M_{2\times4}(\C)\), set
\[
 L_P=\begin{pmatrix}\operatorname{Re}P\\\operatorname{Im}P\end{pmatrix}
 \in M_4(\R).
\]
This is the matrix of the real-linear map
\(\R^4\to\C^2\), \(v\mapsto Pv\), in the ordered coordinates
\((\operatorname{Re}\zeta_1,\operatorname{Re}\zeta_2,
\operatorname{Im}\zeta_1,\operatorname{Im}\zeta_2)^t\).
A \emph{full lattice} is the integer span of a real basis
\cite[Example~1.18, p.~11]{Lee2024}; hence
\[
 P\Z^4\text{ is a full lattice in }\C^2
 \quad\Longleftrightarrow\quad\det L_P\ne0.
\]
In particular, for \(P=(I_2\ Z)\),
\[
 L_{(I_2\ Z)}
 =\begin{pmatrix}I_2&\operatorname{Re}Z\\0&\operatorname{Im}Z\end{pmatrix},
 \qquad \det L_{(I_2\ Z)}=\det\operatorname{Im}Z.
\]

\subsection{The global family as two successive quotients}
\label{subsec:global-periods}
On the logarithmic cover \(q=e^{2\pi it}\), we seek a
holomorphic period matrix \(P(t)\) whose columns correspond
to the ordered basis \(\mathcal B\) and satisfy
\(P(t+1)=P(t)M\). Since \(M\Z^4=\Z^4\), this condition gives
\(P(t+1)\Z^4=P(t)\Z^4\), so the lattice depends on \(q\),
not on the choice of its logarithmic lift. The lattice quotients
can therefore descend to \(\C_q^*\).
The factor involving \(tB\) in \eqref{eq:P} is chosen to
impose this equivariance.
Set
\begin{equation*}
 P_{\mathrm{tw}}(q)=
 \begin{pmatrix}1&-q&i&iq\\q&1&-iq&i\end{pmatrix},
 \qquad q=e^{2\pi it}.
\end{equation*}
The matrix \(P_{\mathrm{tw}}(q)\) describes the standard
twistor family of linear complex structures on \(\R^4\);
see \cite{BuskinIzadi2021} for the general construction
of twistor lines in the period domain of complex tori.
The properties needed here are verified directly below.

\begin{proposition}\label[proposition]{prop:period-bridge}
The matrix
\begin{equation}\label{eq:P}
 P(t)=P_{\mathrm{tw}}(q)
 \begin{pmatrix}I_2&tB\\0&I_2\end{pmatrix}
 =\begin{pmatrix}1&-q&i&(i-t)q\\q&1&-iq&t+i\end{pmatrix}
\end{equation}
has real-linearly independent columns for every \(t\in\C\),
depends holomorphically on \(t\), and satisfies
\begin{equation}\label{eq:equivariance}
 P(t+1)=P(t)M.
\end{equation}
The successive quotients
\[
 \C\times\C^2\longrightarrow
 \widetilde{\mathcal G}=(\C\times\C^2)/\Z^4
 \longrightarrow\mathcal G=\widetilde{\mathcal G}/\Z
\]
define a holomorphic family
\(p:\mathcal G\to\C^*\), with complex two-torus fibres
and marked monodromy \(M\).
\end{proposition}

\begin{proof}
After verifying the period condition, we construct the two
quotients by showing that each group action is holomorphic,
free, and proper.

\noindent\emph{Step 1. The period calculation.}
For \(r=|q|\), direct calculation gives
\begin{equation}\label{eq:global-period-det}
 \det L_{P(t)}
 =(1+r^2)^2-\frac{\log r}{2\pi}(1-r^4)>0,
\end{equation}
because \(-\log r\,(1-r^4)\ge0\) for all \(r>0\).
Every entry of \(P(t)\) is holomorphic in \(t\). The identities
\[
 q(t+1)=q(t),\qquad
 \begin{pmatrix}I_2&(t+1)B\\0&I_2\end{pmatrix}
 =\begin{pmatrix}I_2&tB\\0&I_2\end{pmatrix}M
\]
imply \eqref{eq:equivariance}.

\smallskip
\noindent\emph{Step 2. First quotient: fibrewise lattice translations.}
Set \(\Lambda_t=P(t)\Z^4\), and let
\[
 n\cdot(t,\zeta)=(t,\zeta+P(t)n),\qquad n\in\Z^4.
\]
\smallskip
\noindent\emph{(a) Holomorphicity.}
For fixed \(n\), the entries of \(P(t)n\) are holomorphic.
The inverse transformation is the one defined by \(-n\).

\smallskip
\noindent\emph{(b) Freeness.}
If \(n\cdot(t,\zeta)=(t,\zeta)\), then \(P(t)n=0\).
Since \(P(t)\) is injective as a real-linear map, \(n=0\).

\smallskip
\noindent\emph{(c) Properness.}
Put \(E=\C\times\C^2\). The inverse real-linear map
\(P(t)^{-1}_{\R}:\C^2\to\R^4\) depends smoothly on \(t\).
The action map
\[
 \alpha:\Z^4\times E\longrightarrow E\times E,\qquad
 (n,(t,\zeta))\longmapsto
 \bigl((t,\zeta+P(t)n),(t,\zeta)\bigr)
\]
has closed image
\[
 \operatorname{im}\alpha
 =\left\{\bigl((t,\zeta'),(t,\zeta)\bigr):
        P(t)^{-1}_{\R}(\zeta'-\zeta)\in\Z^4\right\}.
\]
Its inverse on this image is
\[
 \bigl((t,\zeta'),(t,\zeta)\bigr)\longmapsto
 \bigl(P(t)^{-1}_{\R}(\zeta'-\zeta),(t,\zeta)\bigr),
\]
which is continuous. Thus \(\alpha\) is an embedding with closed image.
For any compact subset \(K\subset E\times E\),
\(\alpha^{-1}(K)\cong K\cap\operatorname{im}\alpha\) is compact.
Hence the action is proper.

With these three propertied above, the Holomorphic Quotient Manifold Theorem
\cite[Theorem~1.16]{Lee2024} gives a
Hausdorff complex threefold \(\widetilde{\mathcal G}\) and a
holomorphic covering map onto it. Write its points as
\((t,[\zeta]_t)\), where \([\zeta]_t\in\C^2/\Lambda_t\).
The projection to the first factor is a holomorphic submersion.The smooth
trivialization is given by\
\begin{equation}\label{eq:real-period-trivialization}
 \begin{aligned}
 \C\times\T^4&\longrightarrow\widetilde{\mathcal G},
 &(t,[v])&\longmapsto(t,[P(t)v]_t).
 \end{aligned}
\end{equation}

\smallskip
\noindent\emph{Step 3. Second quotient: translations of the base.}
The formula \eqref{eq:equivariance} and \(M\in\operatorname{SL}(4,\Z)\) give
\[
 P(t+k)=P(t)M^k,\qquad \Lambda_{t+k}=\Lambda_t,\qquad k\in\Z.
\]
Thus the following action is well defined:
\[
 k\cdot(t,[\zeta]_t)=(t+k,[\zeta]_{t+k}).
\]
\smallskip
\noindent\emph{(a) Holomorphicity.}
Since the first quotient map is locally
biholomorphic, the holomorphicity of this action descends from the biholomorphism
\((t,\zeta)\mapsto(t+k,\zeta)\). The inverse is the transformation defined by \(-k\).

\smallskip
\noindent\emph{(b) Freeness.}
If \(k\cdot(t,[\zeta]_t)=(t,[\zeta]_t)\), then \(t+k=t\),
so \(k=0\).

\smallskip
\noindent\emph{(c) Properness.}
For a compact set \(K\subset\widetilde{\mathcal G}\times\widetilde{\mathcal G}\),
we denote \[A=\operatorname{pr}_1(K)\cup\operatorname{pr}_2(K), \quad F_A=\{k\in\Z:k\cdot A\cap A\ne\varnothing\}.\]
If \(k\cdot A\cap A\ne\varnothing\), then
\(k=t(y)-t(x)\) for some \(x,y\in A\).
Since \(t(A)\) is bounded, the set
\(F_A\) is finite.
The inverse image of \(K\) under the action map
\((k,x)\mapsto(k\cdot x,x)\) is a closed subset
of the compact set \(F_A\times A\), so the action is proper.

Again we use the Holomorphic Quotient Manifold Theorem to obtain a complex threefold
\(\mathcal G=\widetilde{\mathcal G}/\Z\) and a holomorphic
covering map. The \(\Z\)-invariant function \(e^{2\pi it}\) descends to
\[
 p:\mathcal G\longrightarrow\C^*,\qquad
 [(t,[\zeta]_t)]\longmapsto e^{2\pi it}.
\]
Both quotient maps are locally biholomorphic and
\(2\pi i e^{2\pi it}\ne0\), so \(p\) is a holomorphic submersion.
On a base chart \(U\subset\C^*\) with a single-valued holomorphic
logarithmic lift \(t=t(q)\),
\[
 U\times\T^4\longrightarrow p^{-1}(U),\qquad
 (q,[v])\longmapsto[(t(q),[P(t(q))v]_{t(q)})]
\]
is a smooth trivialization. Every base point has a neighbourhood
with compact closure in such a chart. The inverse image of this
closure is diffeomorphic to its product with
\(\T^4\), and is therefore compact. A finite covering of any
compact subset of the base by these neighbourhoods proves that
\(p\) is proper.

In the real coordinates \eqref{eq:real-period-trivialization},
the identity
\(
 P(t)v=P(t+1)M^{-1}v
\)
shows that the \(\Z\)-action takes the form
\[
 (t,[v])\longmapsto(t+1,[M^{-1}v]).
\]
The quotient therefore identifies
\[
(t+1,[v])\sim(t,[Mv]).
\]
In particular, over the circle $\{ q=e^{2\pi i\theta} \}$,
the restriction of $\mathcal G$ is a mapping torus 
\[
 \mathcal G\big|_{\{|q|=1\}}
 \cong
 \bigl([0,1]\times\T^4\bigr)
 \big/(1,[v])\sim(0,[Mv]),
\]
as a smooth bundle over $S^1$.

\end{proof}

\section{Proper holomorphic completion of a rank-one degeneration}\label{sec:filling}

This section completes the family constructed in \cref{sec:periods}
across \(0,\infty\in\Pone\). After normalizing the periods near
\(0\), we construct a smooth complex threefold over a disc and
prove that its projection is proper. At each point of the critical
curve, the projection has the local form \(q=xy\).
The same construction applies at \(\infty\).
Finally, identifications with \(\mathcal G\) over the punctured
discs yield the two completed ends to be glued in \cref{sec:global}.

\subsection{The local completion}
\label{subsec:local-completion}

Set
\begin{equation*}
 A(q)=\begin{pmatrix}1&-q\\q&1\end{pmatrix}.
\end{equation*}
Since \(\det A(q)=1+q^2\), for \(q\ne\pm i\), direct multiplication gives
\begin{equation*}
 A(q)^{-1}P(t)=(I_2\ Z(t)).
\end{equation*}
We decompose
\begin{equation}\label{eq:zero-asymptotic}
 Z(t)=tB+C(q),\qquad C(q)=\frac{i}{1+q^2}
       \begin{pmatrix}1-q^2&2q\\-2q&1-q^2\end{pmatrix}.
\end{equation}
Since \(A\) is singular at \(q=\pm i\),
we use this normalization only on
\[
 \Delta_\epsilon=\{|q|<\epsilon\},\quad 0<\epsilon<1.
\]

The normalization is needed only near the omitted point.
The columns of \(P(t)\) remain real-linearly independent over all
of \(\C^*\), including \(q=\pm i\). The following proposition
constructs the filling from the normalized periods
\((I_2\ Z(t))\).

\begin{proposition}\label[proposition]{prop:rank-one-fill}
Let \(Z(t)=tB+C(q)\) be as in \eqref{eq:zero-asymptotic}, with \(q=e^{2\pi it}\). Denote
\[
\Delta^*_\epsilon=\{q \in \C :0<|q|<\epsilon\}, \quad \Delta_\epsilon=\{q \in \C:|q|<\epsilon\}.
\]
For sufficiently small \(0<\epsilon \le \frac{1}{4}\), the matrices
\((I_2\ Z(t))\) define a holomorphic family
\(\mathcal T^*\to\Delta_\epsilon^*\), and there is a complex threefold \(\overline{\mathcal T}\) with a
proper
holomorphic map
\[
 \pi:\overline{\mathcal T}\longrightarrow\Delta_\epsilon,
\]
together with a
biholomorphic isomorphism of families
\[
 \overline{\mathcal T}|_{\Delta_\epsilon^*}
 \cong\mathcal T^*.
\]
The critical locus of \(\pi\) is a complex torus
\(E_i=\C/(\Z+i\Z)\). At every point of this curve there are local
holomorphic coordinates \((z,x,y)\) in which
\begin{equation*}
 q\circ\pi(z,x,y)=xy.
\end{equation*}
The central fibre is \(E_i\times(\Pone/(0\sim\infty))\).
\end{proposition}

\begin{proof}
The proof proceeds in six steps, beginning with uniform estimates
that also apply to the deformation in \cref{sec:relative}.
We then express the regular family in multiplicative coordinates,
construct the node charts, and extend the period action to form
the quotient. The final two steps identify its central fibre and
critical curve and establish properness of the projection to the disc.

Write
\begin{equation*}
 Z(t)=\begin{pmatrix}a(q)&b(q)\\c(q)&t+d(q)\end{pmatrix},
 \quad
 a=d=\frac{i(1-q^2)}{1+q^2},\quad
 b=\frac{2iq}{1+q^2},\quad c=-b.
\end{equation*}
The functions \(a,b,c,d\) are holomorphic for \(|q|<1\), and
\(a(0)=d(0)=i\), \(b(0)=c(0)=0\).

\smallskip
\noindent\emph{Step 1. Uniform estimates.}
Fix, once and for all,
\begin{equation}\label{eq:single-disc-choice}
 0<\epsilon\le\frac14.
\end{equation}
For the parameter family of \cref{sec:relative},
temporarily let \(a,b,c,d\) denote the entries of
\(C(\lambda q)\), where \(|\lambda|\le2\) and \(|q|\le\epsilon\).
Set
\[
 u_j=e^{2\pi i j},\quad
 \ell_j=\log|u_j|=-2\pi\operatorname{Im}j,
 \quad j=a,b,c,d.
\]
Since \(|\lambda q|\le\frac12\),
\[
 |a|=|d|
 \le\frac{1+|\lambda q|^2}{1-|\lambda q|^2}
 \le\frac53,\quad
 |b|=|c|
 \le\frac{2|\lambda q|}{1-|\lambda q|^2}
 \le\frac43.
\]
In particular,
\[
 \operatorname{Im}a = \operatorname{Re} \frac{1-(\lambda q)^2}{1+(\lambda q)^2}
 =\frac{1-|\lambda q|^4}{|1+(\lambda q)^2|^2}
 \ge\frac{1-|\lambda q|^2}{1+|\lambda q|^2}
 \ge\frac35.
\]
With \(C_*:=10\pi/3\) and \(a_*:=6\pi/5\), these inequalities give
\begin{equation*}
\begin{gathered}
    |\ell_j|
 =2\pi|\operatorname{Im}j|
 \le2\pi|j|
 \le\frac{10\pi}{3}=C_*,
 \quad j=a,b,c,d,\\
 \ell_a\le -\frac{6\pi}{5}=-a_*<0
\end{gathered}
\end{equation*}
For \(q\ne0\), define
\begin{equation*}
 L=\log|q|+\ell_d-\frac{\ell_c}{\ell_a}\ell_b.
\end{equation*}
Since \(a=d\) and \(c=-b\),
\[
 L=\log|q|+\ell_a+\frac{\ell_b^2}{\ell_a}.
\]
Direct calculation yields
\begin{equation*}
 \begin{gathered}
 L=\log|q|
   -2\pi\frac{1+|\lambda q|^2}{1-|\lambda q|^2}
   <\log|q|<0,\\
 |L|\ge|\log|q||,\\
 |L-\log|q||\le C_*
 \end{gathered}
\end{equation*}
All these estimates are uniform in \(\lambda\).
For the rest of the proof, set \(\lambda=1\).

\smallskip
\noindent\emph{Step 2. Constructing the regular family from the periods.}
Let
\[
 H_\epsilon=\left\{t\in\C:
       \operatorname{Im}t>\frac{1}{2\pi}\log\frac1\epsilon\right\},
 \quad q=e^{2\pi it}.
\]
Then \(H_\epsilon/\langle S \rangle\cong\Delta_\epsilon^*\), where
\(S:t\mapsto t+1\). For $r=|q|<1$, direct calculation gives
\begin{equation}\label{eq:exact-normalized-det}
 \det\operatorname{Im}Z(t)
 =
 \frac{1}
      {|1+q^2|^2}\big((1+r^2)^2-\frac{\log|q|}{2\pi}(1-r^4)\big)>0.
\end{equation}
Consequently \((I_2\ Z(t)):\R^4\to\C^2\) is a real-linear
isomorphism and \(\widehat\Lambda_t=\Z^2+Z(t)\Z^2\)
is a full lattice, related to the original lattice by
\(\Lambda_t=A(q)\widehat\Lambda_t\).
The lattice action is
\[
 (n,m)\cdot(t,\zeta)
   =(t,\zeta+n+Z(t)m),
 \qquad (n,m)\in\Z^2\times\Z^2.
\]
\smallskip
\noindent\emph{(a) Holomorphicity.}
Since $Z(t)$ is holomorphic, each transformation is
holomorphic, with inverse given by $(-n,-m)$.

\smallskip
\noindent\emph{(b) Freeness.}
If $(n,m)\cdot(t,\zeta)=(t,\zeta)$, then
\[
 n+Z(t)m=0
 \quad\Longrightarrow\quad
 \operatorname{Im}Z(t)m=0
 \quad\Longrightarrow\quad m=0,\ n=0,
\]
because $\det\operatorname{Im}Z(t)>0$.

\smallskip
\noindent\emph{(c) Properness.}
As in the first quotient of
\cref{prop:period-bridge}, the action map is an embedding with closed image, hence proper.

Consequently, the first quotient is a complex threefold by
\cite[Theorem~1.16]{Lee2024}.

Since \(Z(t+1)=Z(t)+B\) and \(B\) is integral,
\(\widehat\Lambda_{t+1}=\widehat\Lambda_t\) and
\[
 (t,[\zeta]_{\widehat\Lambda_t})\longmapsto
       (t+1,[\zeta]_{\widehat\Lambda_{t+1}})
\]
is a well-defined biholomorphism of the first quotient. As in
\cref{prop:period-bridge}, the integer powers act freely
and their action map is an embedding with closed image, hence proper.
Taking this second quotient gives a holomorphic family \(\mathcal T^*\to\Delta_\epsilon^*\).
The real coordinates \(\zeta=(I_2\ Z(t))v\) give smooth local
trivializations with compact fibre \(\R^4/\Z^4\), so the
projection is proper. The resulting two-stage quotient is
\[
 H_\epsilon\times\C^2
 \xrightarrow{(I_2 \ Z(t))\Z^4} (H_\epsilon\times \C^2)/\Z^4
 \xrightarrow{\langle S \rangle} \mathcal T^*.
\]

For the filling at \(0\), we first quotient by the two unit periods
to obtain multiplicative coordinates.
Let \(T_i\) denote translation by the \(i\)-th normalized period,
and let \(S(t,\zeta)=(t+1,\zeta)\). All \(T_i\) commute, and
\begin{equation*}
 ST_i=T_iS\ (i=1,2,3),\quad
 ST_4S^{-1}=T_2^{-1}T_4.
\end{equation*}
Thus after quotienting by the unit periods \(\langle T_1,T_2\rangle\),
the remaining actions commute and \(S\) may be taken before \(T_3,T_4\):
\[
 H_\epsilon\times\C^2
 \xrightarrow{\langle T_1,T_2 \rangle} H_\epsilon\times(\C^*)^2
 \xrightarrow{\langle S \rangle} \mathcal U^*:=(\C^*)^2\times\Delta_\epsilon^*.
\]
We distinguish additive coordinates \(\zeta_j\) from their
multiplicative counterparts
\[
 z_j=e^{2\pi i\zeta_j}\in\C^*,\qquad
 q=e^{2\pi it}\in\Delta_\epsilon^*.
\]
The two remaining generators are
\begin{equation}\label{eq:multiplicative-decks}
 \gamma_1(z_1,z_2)=(u_a z_1,u_c z_2),\quad
 \gamma_2(z_1,z_2)=(u_b z_1,q u_d z_2).
\end{equation}
The transformations \(\gamma_1,\gamma_2\) fix \(q\) and commute.
Thus
\begin{equation*}
 \mathcal T^*=\mathcal U^*/\Gamma,\quad
 \Gamma=\langle\gamma_1,\gamma_2\rangle\cong\Z^2.
\end{equation*}
For a local choice of \(t\), the identification of families is
\[
 \C^2/\widehat\Lambda_t\longrightarrow(\C^*)^2/\Gamma_q,
 \qquad[\zeta_1,\zeta_2]\longmapsto
       [e^{2\pi i\zeta_1},e^{2\pi i\zeta_2}].
\]

\smallskip
\noindent\emph{Step 3. Construction of the complex manifold \(\mathcal V\).}
Set \(\Delta=\Delta_\epsilon\), with \(\epsilon\) fixed in
\eqref{eq:single-disc-choice}. We construct the complex manifold
\(\mathcal V\) and identify its restriction over \(\Delta^*\)
with \(\mathcal U^*\).

\smallskip
\noindent\emph{(a) Node charts and transition maps.}
For each \(k\in\Z\), define
\begin{equation*}
 V_k=\C^*_{z_1}\times
 \{(x_k,y_k,q):x_ky_k=q,\ q\in\Delta\}.
\end{equation*}
The coordinates \((z_1,x_k,y_k)\) identify \(V_k\) with the
open subset \(\{|x_ky_k|<\epsilon\}\) of
\(\C^*\times\C^2\). In particular, \(V_k\) is a complex
threefold, including at \(x_k=y_k=0\).
Define a transition from \(\{x_k\ne0\}\subset V_k\) to
\(\{y_{k+1}\ne0\}\subset V_{k+1}\) by
\begin{equation}\label{eq:adjacent-node-chart}
 x_{k+1}=x_k^2y_k=q x_k,\qquad y_{k+1}=x_k^{-1},
\end{equation}
leaving \(z_1,q\) unchanged. It preserves \(x_{k+1}y_{k+1}=q\)
and has holomorphic inverse
\[
 x_k=y_{k+1}^{-1},\qquad
 y_k=q y_{k+1}=x_{k+1}y_{k+1}^2.
\]

\smallskip
\noindent\emph{(b) The quotient is Hausdorff.}
Let
\[
 \mathcal W=\coprod_{k\in\Z}V_k,
 \qquad Q:\mathcal W\longrightarrow\mathcal V=\mathcal W/\sim,
\]
where \(\sim\) is generated by \eqref{eq:adjacent-node-chart}.
For every open \(O\subset\mathcal W\), the set \(Q^{-1}(Q(O))\)
is a union of images of open subsets of \(O\) under finite
compositions of the transition maps and their inverses.
These images are open, so \(Q(O)\) is open by the quotient topology.
Thus \(Q\) is open.
For \(h\ge1\), points in \(V_k\) and \(V_{k+h}\) are equivalent exactly when
\begin{equation}\label{eq:plumbing-closed-relation}
 z_1=z'_1,\quad
 x_{k+h}=q^h x_k,\quad
 y_k=q^h y_{k+h},\quad
 q^{h-1}x_k y_{k+h}=1.
\end{equation}
These equations and the diagonal relations
are closed on each \(V_k\times V_\ell\), so the equivalence relation
\(R:= \{(p,p') : p \sim p'  \}\subset\mathcal W\times\mathcal W\) is closed.
Then \(\mathcal V\) is Hausdorff by
\cite[Proposition~3.57, p.~68]{Lee2011}.

\smallskip
\noindent\emph{(c) The complex structure and projection.}
Because each \(Q|_{V_k}\) is a homeomorphism onto an open subset, the images of countable bases of the \(V_k\) together form
a countable basis of \(\mathcal V\).
The coordinates \((z_1,x_k,y_k)_k\) on each open manifold \(V_k\) and the biholomorphic
transition maps together determine a complex structure on \(\mathcal V\)
\cite[Proposition~1.1(a)]{Lee2024}. The functions \(x_ky_k\)
agree on overlaps and hence define a holomorphic map
\[
 \widetilde\pi:\mathcal V\longrightarrow\Delta,\qquad
 [(z_1,x_k,y_k,q)_k]\longmapsto q.
\]

\smallskip
\noindent\emph{(d) Identification of the punctured part.}
For \(q\ne0\),
\[
 F:\mathcal V|_{\Delta^*}\longrightarrow\mathcal U^*,\qquad
 [(z_1,x_k,y_k,q)_k]\longmapsto(z_1,q^k y_k,q)
\]
is a well-defined holomorphic map. Define
\begin{equation*}
 \iota:\mathcal U^*\longrightarrow\mathcal V,\qquad
 (z_1,z_2,q)\longmapsto[(z_1,q/z_2,z_2,q)_0].
\end{equation*}
The map \(\iota\) is holomorphic into the chart \(V_0\).
Its representative in \(V_k\) is
\[
 (z_1,x_k,y_k,q)
   =(z_1,q^{k+1}/z_2,q^{-k}z_2,q).
\]
Thus \(F\) and \(\iota\) are holomorphic inverses of each other,
and \(\iota\) is an embedding with
\(\iota(\mathcal U^*)=\mathcal V|_{\Delta^*}\).

\smallskip
\Needspace{5\baselineskip}
\noindent\emph{Step 4. Extending the period action and forming the quotient.}

\smallskip
\noindent\emph{(a) Holomorphicity.}
Under the identification in Step~3(d), the transformations in
\eqref{eq:multiplicative-decks} take the form
\begin{align*}
 \gamma_1:V_k&\longrightarrow V_k,&
 (z_1,x_k,y_k)&\longmapsto
 (u_a z_1,u_c^{-1}x_k,u_c y_k),\\
 \gamma_2:V_k&\longrightarrow V_{k+1},&
 (z_1,x_k,y_k)&\longmapsto
 (u_b z_1,u_d^{-1}x_k,u_d y_k).
\end{align*}
These formulas are holomorphic at \(q=0\), have holomorphic
inverses, and commute with the chart transitions. Moreover,
\(\gamma_1\gamma_2=\gamma_2\gamma_1\). They therefore define
a holomorphic action of \(\Gamma\) on \(\mathcal V\).
Explicitly,
\[
 \gamma_1^m\gamma_2^n[(z_1,x_k,y_k,q)_k]
 =[(u_a^m u_b^n z_1,u_c^{-m}u_d^{-n}x_k,
                         u_c^m u_d^n y_k,q)_{k+n}].
\]
On the regular part,
\[
 F(\gamma_1p)=(u_a z_1,u_c z_2,q),\quad
 F(\gamma_2p)=(u_b z_1,q u_d z_2,q),\quad z_2=q^ky_k.
\]
Thus \(F\) and its inverse \(\iota\) are \(\Gamma\)-equivariant.

\smallskip
\noindent\emph{(b) Logarithmic identities.}
On \(\mathcal V|_{\Delta^*}\), put
\[
 H=\log|z_2|-\frac{\ell_c}{\ell_a}\log|z_1|.
\]
For \(\gamma=\gamma_1^m\gamma_2^n\),
\begin{equation}\label{eq:filling-action-logs}
 \begin{aligned}
 H(\gamma p)&=H(p)+nL,\\
 \log|z_1(\gamma p)|
 &=\log|z_1(p)|+m\ell_a+n\ell_b.
 \end{aligned}
\end{equation}

\smallskip
\noindent\emph{(c) Properness and freeness.}
Consider the action map
\[
 \alpha:\Gamma\times\mathcal V\longrightarrow\mathcal V^2,
 \qquad
 (\gamma,p)\longmapsto(\gamma p,p).
\]
Let \(K\subset\mathcal V \times \mathcal V\) be compact, and put
\[
 A=\operatorname{pr}_1(K)\cup\operatorname{pr}_2(K),
 \quad
 F_A=\{\gamma\in\Gamma:\gamma A\cap A\ne\varnothing\}.
\]
We will show that \(F_A\) is finite. Then
\(\alpha^{-1}(K)\) is a closed subset of the compact
set \(F_A\times A \subset \Gamma \times \mathcal V\) and \(\alpha\) is proper.

Since \(K\) is compact, so is \(A\). Cover \(A\) by finitely
many open neighbourhoods with compact closures contained in
node charts \(V_k\). The coordinate functions
\(z_1,z_1^{-1},x_k,y_k\) are bounded on these closures.
Hence there exist \(N\ge0\) and \(R>1\) such that every point
of \(A\) has a representative satisfying
\[
 |k|\le N,\quad
 R^{-1}\le|z_1|\le R,\quad
 |x_k|,|y_k|\le R.
\]
Suppose
\[
 p'=\gamma_1^m\gamma_2^n p,\qquad p,p'\in A,
\]
and choose such representatives in \(V_k,V_l\),
respectively. Write \(q=q(p)=q(p')\).

If \(q\ne0\),we have
\[
 \log|q|-\log R
 \le\log|y_k|, \quad\log|y_l'|\le\log R.
\]
Thus,
\[
|\log|y_l'|-\log|y_k||
\le|\log|q||+2\log R.
\]
Using \(H(p')-H(p)=nL\), we obtain
\[
 nL
 =(l-k)\log|q|
 +\log|y_l'|-\log|y_k|
 -\frac{\ell_c}{\ell_a}\log|z_1'/z_1|.
\]
Consequently,
\[
 |nL|
 \le(2N+1)|\log|q||
   +2(1+C_*/a_*)\log R.
\]
Since \(|L|\ge|\log|q||\) and \(0<|q|<\epsilon<1\),
\[
 |n|
 \le
 2N+1+
 \frac{2(1+C_*/a_*)\log R}{\log(1/\epsilon)}.
\]
If \(q=0\), the representative of
\(\gamma_1^m\gamma_2^n p\) in \(V_{k+n}\) is equivalent
to that of \(p'\) in \(V_l\). Hence
\eqref{eq:plumbing-closed-relation} gives
\[
 |k+n-l|\le1,
 \qquad\text{so}\qquad |n|\le2N+1.
\]
Thus \(n\) is uniformly bounded in both cases.

For every \(q\), we have
\[
 m\ell_a+n\ell_b=\log|z_1'/z_1|.
\]
Using \(|\ell_a|\ge a_*\) and \(|\ell_b|\le C_*\),
we conclude that
\[
 |m|\le\frac{2\log R+C_*|n|}{a_*}.
\]
Therefore only finitely many integer pairs \((m,n)\)
can occur, so \(F_A\) is finite and the action is proper.
For each \(p\in\mathcal V\),
\[
 \alpha^{-1}(\{(p,p)\})=\Gamma_p\times\{p\},
\]
where \(\Gamma_p=\{\gamma\in\Gamma:\gamma p=p\}\) is the stabilizer of the action.
Since \(\alpha\) is proper, \(\Gamma_p\) must be compact, so finite.
\(\Gamma\cong\Z^2\) then implies \(\Gamma_p\) is trivial.
Thus the action is also free.

The Holomorphic Quotient Manifold Theorem
\cite[Theorem~1.16]{Lee2024} now gives a Hausdorff complex
threefold \(\overline{\mathcal T}:=\mathcal V/\Gamma\) and a holomorphic covering
\[
 \rho:\mathcal V\longrightarrow
 \overline{\mathcal T}=\mathcal V/\Gamma.
\]
Since \(q\) is \(\Gamma\)-invariant, it descends to a holomorphic map
\(\pi:\overline{\mathcal T}\to\Delta\).
The equivariant biholomorphism \(\iota\) induces a holomorphic
embedding over \(\Delta^*\):
\[
 j:\mathcal T^*=\mathcal U^*/\Gamma\longrightarrow
 \overline{\mathcal T}=\mathcal V/\Gamma,
 \quad j(\mathcal T^*)=\overline{\mathcal T}|_{\Delta^*}.
\]

\smallskip
\noindent\emph{Step 5. The central fibre and critical locus.}
In each chart,
\[
 \widetilde\pi^{-1}(0)\cap V_k
 =
 \C^*_{z_1}\times
 \bigl(\{x_k=0\}\cup\{y_k=0\}\bigr).
\]
The gluing at \(q=0\) is
\[
 (x_k,0)_k\sim(0,x_k^{-1})_{k+1},
 \qquad x_k\ne0,
\]
with \(z_1\) unchanged. Thus the branch \(\{y_k= 0\} \subset V_k\)
and the branch \(\{x_{k+1} = 0\} \subset V_{k+1}\) form a copy of
\(\C^*\times\Pone\), giving
\[
 \widetilde\pi^{-1}(0)
 \cong
 \C^*\times
 \bigl(\cdots\cup\Pone\cup\Pone\cup\Pone\cup\cdots\bigr),
\]
where adjacent spheres meet at their marked points
\(0\) and \(\infty\).

At \(q=0\), the coefficients satisfy
\[
 u_a=u_d=\kappa=e^{-2\pi},
 \qquad u_b=u_c=1.
\]
Hence the two generators act by
\[
 \begin{aligned}
 \gamma_1(z_1,x_k,y_k)_k
   &=(\kappa z_1,x_k,y_k)_k,\\
 \gamma_2(z_1,x_k,y_k)_k
   &=(z_1,\kappa^{-1}x_k,\kappa y_k)_{k+1}.
 \end{aligned}
\]
Taking the quotient by \(\gamma_1\) replaces the \(\C^*\) factor by
\[
 \C^*/\kappa^\Z
 \cong E_i=\C/(\Z+i\Z).
\]
The remaining generator \(\gamma_2\) identifies each sphere with the
next one. Hence
\[
 F_0:=\pi^{-1}(0)
 \cong
 \frac{E_i\times\Pone}
      {(p,0)\sim(p,\infty)}
 =
 E_i\times\bigl(\Pone/(0\sim\infty)\bigr).
\]
In particular, \(F_0\) is compact.

Since \(\rho\) is locally biholomorphic, the map \(\pi\) has
the same local expression \(q=xy\). Therefore
\[
 d(xy)=y\,dx+x\,dy,
 \quad
 \operatorname{Crit}(\pi)
 =
 \rho\!\left(\bigcup_{k\in\Z}\{x_k=y_k=0\}\right)
 \cong E_i.
\]
Here \(\gamma_2\) identifies all nodes of the chain.

\smallskip
\noindent\emph{Step 6. Properness over the disc.}
Properness of the group action gives the Hausdorff quotient, but
does not by itself prove properness of \(\pi\). We prove the
latter using the compact central fibre and the connected regular
fibres.
The compact central fibre has a relatively compact open
neighbourhood in \(\overline{\mathcal T}\):
\[
 F_0\subset\Omega\Subset\overline{\mathcal T}.
\]
The map \(\pi\) is open. At a regular point this follows from the
submersion theorem; at a critical point the local expression
\((z,x,y)\mapsto xy\) maps every sufficiently small product
neighbourhood onto a neighbourhood of \(0\).
Moreover, \(\partial\Omega\) is compact and disjoint from \(F_0\).
Thus \(\pi(\partial\Omega)\) is closed and does not contain \(0\),
whereas \(\pi(\Omega)\) is an open neighbourhood of \(0\).
Consequently, there is a \(0<\delta<\epsilon\) satisfying
\[
 \overline\Delta_\delta
 \subset
 \pi(\Omega)\setminus\pi(\partial\Omega).
\]
For \(0<|q|\le\delta\), the set \(\Omega\cap\pi^{-1}(q)\)
is nonempty and open in the connected torus \(\pi^{-1}(q)\).
It is also closed there, because
\(\pi^{-1}(q)\cap\partial\Omega=\varnothing\).
Hence every such fibre lies in \(\Omega\).
Together with \(F_0\subset\Omega\), we have
\[
 \pi^{-1}(\overline\Delta_\delta)\subset\Omega.
\]
\(\pi^{-1}(\overline\Delta_\delta)\) is closed in \(\overline{\mathcal T}\)
and contained in the compact set \(\overline\Omega\), so it is compact.
Thus \(\pi\) is proper near \(0\).
Together with properness over \(\Delta_\epsilon^*\), and the
locality of properness on the base, this proves properness over
\(\Delta_\epsilon\).
The same compact-neighbourhood argument applies to a continuous
open map between manifolds of any dimensions, provided the fibre
under consideration is nonempty and compact and nearby fibres
are connected.

The proof is now complete. 
\end{proof}

\begin{remark}\label[remark]{rem:infinity-filling}
\label[remark]{rem:marking-not-map}
We apply the same construction at \(\infty\). Replacing the period
matrix \((I_2\ Z(t))\) by \((I_2\ {-Z(t)})\)
does not change its lattice. Set \(w=q^{-1}\) and \(s=-t\),
so that \(w=e^{2\pi is}\). The last two columns then become
\begin{equation*}
 Z_\infty(s):=-Z(-s)=sB+C(-w).
\end{equation*}
\begin{equation*}
 \begin{gathered}
 C(-w)=-C(w^{-1})
 =\frac{i}{1+w^2}
       \begin{pmatrix}1-w^2&-2w\\2w&1-w^2\end{pmatrix}.
 \end{gathered}
\end{equation*}
With \((t,q,C(q))\) replaced by \((s,w,C(-w))\), the determinant
\eqref{eq:exact-normalized-det} and the uniform estimates apply with
\(\lambda=-1\). The same radius, charts, and action argument therefore
give the filling at \(\infty\)
\[
 \pi_\infty:\overline{\mathcal T}_\infty\longrightarrow
 \Delta_\epsilon,
 \qquad \overline{\mathcal T}_\infty|_{\Delta_\epsilon^*}
       \cong\mathcal T_\infty^*,
\]
with periods \((I_2\ Z_\infty(s))\).
\end{remark}

This construction is related to the logarithmic description of
degenerating complex tori in
\cite{KajiwaraKatoNakayama2008}; the Tate model in Section~1.1.3
of that paper has the same chain of rational curves.
The proof above uses ordinary holomorphic charts and estimates
for the discrete action. It does not invoke their proper-model
theorem or require a polarization.

\subsection{Comparison with the original torus family}
\label{subsec:period-comparison}

Reversing the coordinate normalization identifies the punctured
fillings with the corresponding restrictions of \(\mathcal G\).
The change of integral marking must be recorded in this comparison,
since it enters the attaching map in \cref{sec:recognition}.

With the radius \(\epsilon\) chosen above, the local discs and
base charts are
\[
 \begin{aligned}
 \Delta_0&=\{|q|<\epsilon\},
 &U_0&=\Pone\setminus\{\infty\},\\
 \Delta_\infty&=\{|w|<\epsilon\},
 &U_\infty&=\Pone\setminus\{0\}.
 \end{aligned}
\]
Write \(O_j\subset U_j\) for the image of \(\Delta_j^*\)
under the corresponding base identification:
\begin{equation*}
 \begin{aligned}
 \Delta_0^*&\longrightarrow
 O_0=\{0<|q|<\epsilon\},
 &q&\longmapsto q,\\
 \Delta_\infty^*&\longrightarrow
 O_\infty=\{|q|>\epsilon^{-1}\},
 &w&\longmapsto q=w^{-1}.
 \end{aligned}
\end{equation*}
The two fillings constructed above are
\[
 \pi_j:\overline{\mathcal T}_j\longrightarrow\Delta_j,
 \qquad j=0,\infty.
\]

\begin{corollary}
\label[corollary]{cor:both-fillable}
The horizontal maps in the following commutative diagram are
holomorphic isomorphisms:
\begin{equation}\label{eq:comparison-composite}
 \begin{tikzcd}[column sep=large,row sep=large]
 \overline{\mathcal T}_j|_{\Delta_j^*}
   \arrow[r,"\cong"]
   \arrow[d,"\pi_j"']
 &
 \mathcal T_j^*
   \arrow[r,"{\substack{[\zeta]\mapsto[A(q)\zeta]\\ [\zeta]\mapsto[A(w^{-1})\zeta]}}"',"\cong"]
   \arrow[d]
 &
 \mathcal G|_{O_j}
   \arrow[d,"p"]
 \\
 \Delta_j^*
   \arrow[r,equal]
 &
 \Delta_j^*
   \arrow[r,
     "{\substack{q\mapsto q\\
                  w\mapsto w^{-1}}}"']
 &
 O_j .
 \end{tikzcd}
\end{equation}
Here \(\zeta\in\C^2\), and each equivalence class is taken with
respect to the period lattice of the indicated fibre.
Gluing along these identifications gives proper holomorphic maps
\[
 \overline p_j:\overline{\mathcal G}_j\longrightarrow U_j
\]
from complex threefolds, with the singular fibres
and local models of \cref{prop:rank-one-fill}.
\end{corollary}

\begin{proof}
The filling construction supplies the first horizontal
identification. For the second, the logarithmic coordinates at
the respective ends satisfy
\[
 q=e^{2\pi it}\quad\text{near }0,\quad
 w=e^{2\pi is},\quad q=w^{-1}=e^{-2\pi is}
 \quad\text{near }\infty.
\]
The period identities are
\begin{align}
 P(t)&=A(q)(I_2\ Z(t)),
 \label{eq:zero-period-comparison}\\
 P(-s)G_0&=A(w^{-1})(I_2\ Z_\infty(s)).
 \label{eq:infinity-period-comparison}
\end{align}
Together with \(G_0\Z^4=\Z^4\), these identities show that
\(A(q)\) and \(A(w^{-1})\) identify the corresponding period lattices.
On their respective punctured discs these matrices are
single-valued, holomorphic, and invertible, since
\[
 \det A(q)=1+q^2\ne0\quad\text{on }\Delta_0^*,
 \qquad
 \det A(w^{-1})=1+w^{-2}\ne0\quad\text{on }\Delta_\infty^*.
\]
They therefore induce the biholomorphisms in
\eqref{eq:comparison-composite}, with inverses induced by
the inverse matrices over the inverse base maps.

It follows from equations
\eqref{eq:zero-period-comparison}--%
\eqref{eq:infinity-period-comparison}
that the changes of real period coordinates are:
\[
 v\longmapsto v\quad(j=0),\qquad
 v\longmapsto G_0v\quad(j=\infty).
\]

The diagram \eqref{eq:comparison-composite} determines
identification relations, denoted by \(\sim\), on the total
spaces and on their bases. For two copies
\(\mathcal G^{(0)},\mathcal G^{(\infty)}\) of the original family, the
\emph{unshifted} completions are
\begin{equation}\label{eq:one-end-completions}
 \overline{\mathcal G}_j
 =
 \bigl(\overline{\mathcal T}_j
       \sqcup\mathcal G^{(j)}\bigr)/{\sim},
 \quad
 \overline p_j:\overline{\mathcal G}_j
 \longrightarrow
 \bigl(\Delta_j\sqcup\C_q^*\bigr)/{\sim}=U_j.
\end{equation}
Under the induced inclusions of the base charts into \(U_j\),
\[
 \begin{aligned}
 \overline p_j^{-1}(\Delta_j)
 &\cong\overline{\mathcal T}_j,\\
 \overline p_j^{-1}(\C_q^*)
 &\cong\mathcal G.
 \end{aligned}
\]
Points over distinct base points are separated by the inverse images of disjoint open neighbourhoods in the base, 
while points over the same base point lie in a common open
manifold. Thus \(\overline{\mathcal G}_j\) is Hausdorff.
The countable holomorphic atlases of the two pieces are compatible
under the gluing biholomorphism, so they define a complex threefold.
Properness follows from that of \(p\) and \(\pi_j\), since it
is local on the base.

The local identification over \(\Delta_j\) preserves the added
fibre and critical curve:
\[
 \overline p_0^{-1}(0)
 \cong
 \overline p_\infty^{-1}(\infty)
 \cong
 E_i\times\bigl(\Pone/(0\sim\infty)\bigr),
 \qquad
 \operatorname{Crit}(\overline p_j)\cong E_i.
\]
Along the critical curves, the local equations remain
\[
 q\circ\overline p_0=xy,
 \qquad
 w\circ\overline p_\infty=xy.
\]

\end{proof}

We remark that the canonical open embeddings
\[
 \jmath_j:\mathcal G\hookrightarrow\overline{\mathcal G}_j,
 \qquad
 \jmath_j(\mathcal G)=\overline p_j^{-1}(\C_q^*)
\]
identify \(\mathcal G\) with the regular part of each completion.

\section{Relative smooth splitting of the filled ends}
\label{sec:relative}

The preceding section gives the two ends as complex threefolds.
Determining the smooth manifold obtained by gluing them requires
product identifications that also specify the boundary coordinates.
To do it, we deform the periods to a product family and construct a flow
preserving the base coordinate. Near the cutting circle, the flow
keeps the full real period coordinates fixed. This induces the
diffeomorphisms \(D_0,D_\infty\) used to calculate the attaching map
in \cref{sec:recognition}. These identifications are smooth, not necessarily holomorphic.

Choose the compact base discs 
\[
 K_0=\{|q|\le1\},\quad
 K_\infty=\{|w|\le1\}
\]
for \(\overline{\mathcal G}_0\) and
\(\overline{\mathcal G}_\infty\), respectively. Write
\(S=\partial K_0\cap\partial K_\infty=\{|q|=1\}\)
for their common boundary, identifying the two boundary circles
by \(q\mapsto w=q^{-1}\).
In the following, the Tate blocks \(N_0\to K_0\) and
\(N_\infty\to K_\infty\) are taken over the whole closed
discs, extending the local models of
\cref{subsec:node-topology} by their regular Tate
families on the outer annuli.

\subsection{Deformation with fixed boundary coordinates}
\label{subsec:relative-deformation}

We begin with the end at \(0\). The deformation below replaces
\(P_{\mathrm{tw}}(q)\) by \(P_{\mathrm{tw}}(\lambda q)\), leaving
the monodromy unchanged. At \(\lambda=0\), the family is the product of the fixed complex torus \(E_i\), of complex dimension \(1\), and the Tate family.

\begin{proposition}\label[proposition]{prop:parametric}
The period matrices
\begin{equation*}
 P_\lambda(t)
 =P_{\mathrm{tw}}(\lambda q)
   \begin{pmatrix}I_2&tB\\0&I_2\end{pmatrix},
 \quad q=e^{2\pi it},\quad 0\le\lambda\le1,
\end{equation*}
admit simultaneous compact completions
\(\mathcal B_\lambda\to K_0\), with
\[
 \mathcal B_1=\overline{\mathcal G}_0|_{K_0},
 \qquad
 \mathcal B_0\cong E_i\times N_0,
 \qquad E_i=\C/(\Z+i\Z).
\]
Here \(N_0\) is the positive Tate block with multiplier
\(\kappa q\), where \(\kappa=e^{-2\pi}\).

There is a base-preserving diffeomorphism
\[
 H:\mathcal B_0\xrightarrow{\;\cong\;}\mathcal B_1
\]
which extends across \(S\).
On an open annulus about \(S\), it satisfies
\[
 H(q,[P_0(t)v])=(q,[P_1(t)v]),
 \qquad v\in\R^4/\Z^4.
\]
\end{proposition}

\begin{proof}
We first construct a simultaneous filling and show that its
projection is proper. The flow of a suitable vector field then
gives the required diffeomorphism while preserving the prescribed
coordinates near the boundary.

\noindent\emph{Step 1. The parameter family and its filling.}

\smallskip
\noindent\emph{(a) The regular family.}
The matrices \(P_\lambda(t)\) depend holomorphically on
\((\lambda,t)\) and satisfy
\[
 P_1=P,\qquad
 P_0(t)=
 \begin{pmatrix}
  1&0&i&0\\
  0&1&0&t+i
 \end{pmatrix},
 \qquad
 P_\lambda(t+1)=P_\lambda(t)M.
\]
For any \(\lambda\in\C\),
\begin{equation}\label{eq:uniform-det}
 \det L_{P_\lambda(t)}
 =(1+|\lambda q|^2)^2-\frac{\log |q|}{2\pi}(1-|\lambda q|^4).
\end{equation}
For \(|\lambda|\le2\) and
\(0<|q|\le\epsilon\le\frac14\), we have
\(|\lambda q|\le\frac12\) and \(\log|q|<0\). Hence
\(
\det L_{P_\lambda(t)} >0.
\)
On the compact parameter region
\(
 \lambda\in[0,1],\, \epsilon\le |q|\le1,
\)
we likewise have \(|\lambda q|\le1\), \(\log|q|\le0\), and hence
\(\det L_{P_\lambda(t)}\ge1\).
Since the expression in \eqref{eq:uniform-det} is continuous
in \((\lambda,|q|)\), there exist
\[
 [0,1]\subset U\Subset\{|\lambda|<2\},
 \quad 0<\eta<\frac12,
\]
with \(U\) open, such that for any \(\lambda\in U,\ \epsilon\le|q|<1+\eta\),
\(
 \det L_{P_\lambda(t)}\ge\frac12.
\)
Together, the estimates on these two regions show that for
\(\lambda\in U\) and \(0<|q|<1+\eta\),
\[
 \det L_{P_\lambda(t)}\ge\frac12.
\]

The two quotient constructions in \cref{prop:period-bridge}
therefore apply with \(\lambda\) retained as a parameter. They give
a proper holomorphic family
\[
 \mathfrak B^*
 \longrightarrow
 U\times\{0<|q|<1+\eta\}.
\]
Both actions fix \(\lambda\).

\smallskip
\noindent\emph{(b) The simultaneous local filling.}
On the disc already fixed in
\cref{prop:rank-one-fill},
\[
 |\lambda q|\le\tfrac12
 \quad\Longrightarrow\quad
 \det A(\lambda q)=1+(\lambda q)^2\ne0.
\]
Hence the normalized periods are
\[
 A(\lambda q)^{-1}P_\lambda(t)
   =(I_2\ \ tB+C(\lambda q)).
\]
The plumbing space \(\mathcal V\) is independent of \(\lambda\).
Replacing \(C(q)\) by \(C(\lambda q)\) in the action of
\cref{prop:rank-one-fill} gives
\[
 \Gamma\curvearrowright U\times\mathcal V,
 \qquad
 \mathfrak T=(U\times\mathcal V)/\Gamma,
\]
where every element of \(\Gamma\) fixes \(\lambda\).
The coefficients depend holomorphically on \((\lambda,q)\),
and the estimates for \(C(\lambda q)\) in
\cref{prop:rank-one-fill} hold uniformly for \(\lambda\in U\),
so the same compact-set argument proves that the joint action
on \(U\times\mathcal V\) is holomorphic and proper.
Therefore the stabilizers are finite, and \(\Gamma \cong \Z^2\) again implies the action is free.
The Holomorphic Quotient Manifold Theorem
\cite[Theorem~1.16]{Lee2024} then makes \(\mathfrak T\)
a complex fourfold.
Let \(f:\mathfrak T\to U\times\Delta_\epsilon\) be the projection.
Since the action at \(q=0\) is independent of \(\lambda\),
\[
 f^{-1}(\lambda_0,0)
 \cong E_i\times\bigl(\Pone/(0\sim\infty)\bigr)
 \qquad(\lambda_0\in U).
\]
These fibres are compact and connected, and the other fibres are
connected complex tori. The local models
\[
 (\lambda,u,x,y)\longmapsto(\lambda,xy)
\]
and the regular submersion charts show that \(f\) is open.

The relatively compact neighbourhood argument in Step~6
of the proof of \cref{prop:rank-one-fill} therefore applies
near each \((\lambda_0,0)\).
Together with properness over \(U\times\Delta_\epsilon^*\),
this proves that \(f\) is proper.

\smallskip
\noindent\emph{(c) Comparison and gluing.}
On \(U\times\Delta_\epsilon^*\), the identity
\[
 P_\lambda(t)
 =A(\lambda q)\bigl(I_2\;\;tB+C(\lambda q)\bigr)
\]
induces the biholomorphic isomorphism
\[
 \mathfrak T|_{U\times\Delta_\epsilon^*}
 \longrightarrow
 \mathfrak B^*|_{U\times\Delta_\epsilon^*},
 \qquad
 (\lambda,q,[\zeta])
 \longmapsto
 (\lambda,q,[A(\lambda q)\zeta]).
\]
It covers the identity on \(U\times\Delta_\epsilon^*\).
The comparison and gluing argument of
\cref{subsec:period-comparison}, with \(\lambda\)
retained as a parameter, gives a Hausdorff complex fourfold
\[
 \mathfrak B^{\mathrm{an}}
 =
 \mathfrak T
 \underset{U\times\Delta_\epsilon^*}{\cup}
 \mathfrak B^*
\]
and a proper holomorphic map
\[
 \Pi:\mathfrak B^{\mathrm{an}}
 \longrightarrow U\times\Delta_{1+\eta},
 \quad \Pi=(\lambda,q).
\]

\smallskip
\noindent\emph{Step 2. A flow with prescribed boundary coordinates.}
The map to the base disc has critical points, so we do not apply
the submersion theorem to that map. Instead, we lift the parameter
vector field \(\partial_\lambda\), keeping \(q\) fixed. The node
equation \(q=xy\), which is independent of \(\lambda\), permits
such a lift also along the critical locus.

For an open interval
\([0,1]\subset I\Subset U\cap\R\), the parameter slices are
\[
 \mathcal B_\lambda
   =\Pi^{-1}(\{\lambda\}\times K_0),
 \qquad
 \mathfrak B=\Pi^{-1}([0,1]\times K_0).
\]
Properness of \(\Pi\) makes \(\mathfrak B\) compact.
The vector fields below are defined on a neighbourhood of
\(\mathfrak B\) in the smooth manifold without boundary
\(\Pi^{-1}(I\times\Delta_{1+\eta})\).

Choose an open covering \(U_a\) of
\(\Pi^{-1}(I\times\Delta_{1+\eta})\) by node charts with coordinates
\((\lambda,u,x,y)\), \(q=xy\), and regular charts with coordinates
\((\lambda,q,\text{fibre coordinates})\).
In each chart, the vector field \(V_a=\partial_\lambda\) satisfies
\[
 d\lambda(V_a)=1,\quad dq(V_a)=0
\]
and keeps the remaining coordinates fixed.

On an open annulus about \(S=\partial K_0\), the real
period coordinates are
\begin{equation*}
 \begin{gathered}
 q=re^{2\pi i\theta},\qquad
 t=\theta+\frac{\log r}{2\pi i},\qquad
 \zeta=P_\lambda(t)v,\\
 1-\eta'<r<1+\eta',\qquad
 0<\eta'<\min\{1,\eta\},\qquad
 \theta\in\R,\quad v\in\R^4,\\
 v\sim v+n\quad(n\in\Z^4),\qquad
 (\theta+1,v)\sim(\theta,Mv).
 \end{gathered}
\end{equation*}
The field
\[
 V_{\mathrm{ann}}
   =\partial_\lambda|_{(r,\theta,v)}
\]
descends because these identifications are independent
of \(\lambda\).

Choose a partition of unity subordinate to this covering and the
annular neighbourhood, and define
\[
 V=\chi_{\mathrm{ann}}V_{\mathrm{ann}}
       +\sum_a\chi_aV_a,
 \quad
 \chi_{\mathrm{ann}}+\sum_a\chi_a=1,
\]
where \(\chi_{\mathrm{ann}}=1\) over a smaller annulus about
\(S\), uniformly for \(0\le\lambda\le1\). Thus
\begin{equation*}
 d\lambda(V)=1,\quad dq(V)=0,\quad
 V=V_{\mathrm{ann}}\quad\text{over the smaller annulus}.
\end{equation*}

For \(p\in\mathcal B_{\lambda_0}\), the flow satisfies
\[
 \lambda(\operatorname{Fl}_s^V(p))=\lambda_0+s,
 \quad
 q(\operatorname{Fl}_s^V(p))=q(p)
\]
wherever defined. Its trajectory therefore remains in the compact set
\(\mathfrak B\) while \(0\le\lambda_0+s\le1\).
By the escape lemma
\cite[Lemma~9.19]{Lee2013}, the flow exists for
\(-\lambda_0\le s\le1-\lambda_0\).
The time-one map is therefore a diffeomorphism
\[
 H:=\operatorname{Fl}_1^V:{\mathcal B_0} \longrightarrow {\mathcal B_1}
\]
with inverse given by the time-minus-one map. Since the flow
preserves \(q\), the following diagram commutes:
\[
 \begin{tikzcd}[column sep=large,row sep=large]
 \mathcal B_0
   \arrow[r,"H","\cong"']
   \arrow[d,"q"']
 &
 \mathcal B_1
   \arrow[d,"q"]\\
 K_0\arrow[r,equal]&K_0.
 \end{tikzcd}
\]

On the prescribed annulus,
\[
 (\dot\lambda,\dot r,\dot\theta,\dot v)=(1,0,0,0),
\]
so
\[
 \operatorname{Fl}_s^V
 \bigl(\lambda_0,q,[P_{\lambda_0}(t)v]\bigr)
 =
 \bigl(\lambda_0+s,q,[P_{\lambda_0+s}(t)v]\bigr).
\]
In particular,
\[
 H(q,[P_0(t)v])=(q,[P_1(t)v]).
\]
Since the annular prescription holds for every
\(0\le\lambda\le1\), this formula also defines the flow
on the outer half of the annulus. It therefore extends
\(H\) across \(S\), with inverse obtained by interchanging
\(P_0\) and \(P_1\).

The comparison
\eqref{eq:one-end-completions} identifies
\[
 \mathcal B_1=\overline{\mathcal G}_0|_{K_0}.
\]
At \(\lambda=0\), the periods and the filling action split
into two factors. The resulting identification
\(\mathcal B_0\cong E_i\times N_0\) is written explicitly below.

\end{proof}

\subsection{Product identifications at the two ends}
\label{subsec:relative-products}

We now combine the flow with the product identification at
\(\lambda=0\). The construction below gives the formula on the
regular part, verifies its compatibility with the node charts,
and specifies the boundary coordinates required in \cref{sec:recognition}.

\begin{proposition}\label[proposition]{prop:relative-tate}
There is a base-preserving diffeomorphism
\[
 \begin{aligned}
 D_0:\overline{\mathcal G}_0|_{K_0}
   &\xrightarrow{\;\cong\;}E_i\times N_0.
 \end{aligned}
\]
It extends across the cutting circle. On an open annulus about
this circle, it preserves the full real period markings
\(v\) at \(0\).
\end{proposition}

\begin{proof}
\noindent\emph{Step 1. The product identification at \(\lambda=0\).}
At \(\lambda=0\), the periods \(P_0(t)\) from
\cref{prop:parametric} give the product
\[
 \C^2/P_0(t)\Z^4
 =
 \frac{\C}{\Z+i\Z}
 \times
 \frac{\C}{\Z+(t+i)\Z}.
\]
To distinguish additive and multiplicative equivalence classes,
we use the notation
\[
 [\zeta]_\Lambda=\zeta+\Lambda,
 \qquad
 [u]_{\alpha^\Z}=\{\alpha^nu:n\in\Z\}.
\]
Using the exponential coordinates of
\cref{subsec:node-topology} and the identification
\(E_i\cong\C^*/\kappa^\Z\), the product map on the regular part is
\[
 \begin{aligned}
 J_{\mathrm{reg}}:
 \mathcal B_0|_{K_0\setminus\{0\}}
 &\longrightarrow
 E_i\times N_0|_{K_0\setminus\{0\}},\\
 \bigl(q,[(\zeta_1,\zeta_2)]_{P_0(t)\Z^4}\bigr)
 &\longmapsto
 \bigl(
 [e^{2\pi i\zeta_1}]_{\kappa^\Z},
 (q,[e^{2\pi i\zeta_2}]_{(\kappa q)^\Z})
 \bigr).
 \end{aligned}
\]

It remains to extend this product through the central fibre.
Let \(\mathcal W_{\mathrm{node}}\) be the complex surface obtained
by gluing the charts \(\{x_ky_k=q,\ q\in\Delta_\epsilon\}\)
using \eqref{eq:adjacent-node-chart}. Since the transitions leave
the \(\C^*\)-coordinate fixed, the local plumbing space is
\[
 \mathcal V=\C^*\times\mathcal W_{\mathrm{node}},
\]
and at \(\lambda=0\) the two generators \(\gamma_1,\gamma_2\) act on these two factors separately:
\[
 \begin{aligned}
 \gamma_1(u,\omega)&=(\kappa u,\omega),\\
 \gamma_2(u,\omega)&=(u,\gamma\omega),\\
 \gamma[(x,y,q)_k]
   &=[(\kappa^{-1}x,\kappa y,q)_{k+1}].
 \end{aligned}
\]
Their quotient therefore has the product identification
\[
 \begin{aligned}
 J_{\mathrm{loc}}:
 \frac{\C^*\times\mathcal W_{\mathrm{node}}}
      {\langle\gamma_1,\gamma_2\rangle}
 &\xrightarrow{\;\cong\;}
 \frac{\C^*}{\kappa^\Z}
 \times
 \frac{\mathcal W_{\mathrm{node}}}{\langle\gamma\rangle},\\
 [(u,\omega)]_\Gamma
 &\longmapsto
 \bigl([u]_{\kappa^\Z},
       [\omega]_{\langle\gamma\rangle}\bigr).
 \end{aligned}
\]
Here the second factor is the local Tate filling.

The two product maps agree under the filling comparison.
Indeed, since \(A(0)=I_2\), the corresponding representatives
on the punctured overlap are
\[
 \bigl(q,[(\zeta_1,\zeta_2)]_{P_0(t)\Z^4}\bigr)
 \longleftrightarrow
 [(u,\omega)]_\Gamma,
 \qquad
 \begin{cases}
 u=e^{2\pi i\zeta_1},\\
 \omega=[(qe^{-2\pi i\zeta_2},e^{2\pi i\zeta_2},q)_0].
 \end{cases}
\]
For any chart representative
\(\omega=[(x_k,y_k,q)_k]\) of this same point, the
chart-independent coordinate satisfies
\[
 q^ky_k=e^{2\pi i\zeta_2}.
\]
Moreover, the action of \(\gamma\) becomes
\[
 \gamma:\ q^ky_k
 \longmapsto q^{k+1}\kappa y_k
 =\kappa q(q^ky_k).
\]
Hence the corresponding quotient classes are related by
\[
 [\omega]_{\langle\gamma\rangle}
 \longleftrightarrow
 (q,[q^ky_k]_{(\kappa q)^\Z}),
\]
so \(J_{\mathrm{loc}}\) agrees with \(J_{\mathrm{reg}}\) on
the punctured overlap. They bring a base-preserving product
identification on the whole block:
\[
 J:\mathcal B_0\xrightarrow{\;\cong\;}E_i\times N_0.
\]

\smallskip
\noindent\emph{Step 2. The diffeomorphism for the end at \(0\).}
We now have a product identification at \(\lambda=0\).
The backward flow from \cref{prop:parametric}
takes the original completed block to this endpoint, so the
required diffeomorphism is the composition
\begin{equation*}
 \begin{tikzcd}[column sep=large,row sep=large]
 \overline{\mathcal G}_0|_{K_0}=\mathcal B_1
   \arrow[r,"H^{-1}"]
   \arrow[rr,bend left=25,"D_0"]
   \arrow[d,"q"']
 &
 \mathcal B_0
   \arrow[r,"J"]
   \arrow[d,"q"]
 &
 E_i\times N_0
   \arrow[d,"q"]\\
 K_0\arrow[r,equal]&K_0\arrow[r,equal]&K_0.
 \end{tikzcd}
\end{equation*}
In other words,
\[
 D_0=J\circ H^{-1}.
\]

On the annulus about the cutting circle, \(H^{-1}\) preserves
the real period coordinate \(v\). For
\(v=(x,\mu,y,\nu)^{\mathsf T}\), this gives
\[
 \begin{aligned}
 (q,[P(t)v])
 &\xmapsto{H^{-1}}(q,[P_0(t)v]),\\
 P_0(t)v
 &=\bigl(x+iy,\ \mu+(t+i)\nu\bigr)^{\mathsf T}.
 \end{aligned}
\]
By \eqref{eq:real-tate-coordinates}, the second component
corresponds to \((r,[\theta,\mu,\nu])\). Writing \(E_i\) in
additive coordinates, we obtain
\begin{equation}\label{eq:whole-block-collar}
 D_0(q,[P(t)v])
 =
 \bigl(
 [x+iy]_{\Z+i\Z},
 (r,[\theta,\mu,\nu])
 \bigr).
\end{equation}
This formula holds on an open annulus about \(S\).
Away from that annulus, the same map \(D_0=J\circ H^{-1}\)
is defined using the backward flow on the rest of the block.

\end{proof}

\begin{remark}
\label[remark]{rem:infinity-product}
We also need the product identification at \(\infty\), including
its change of marking. With \(w=q^{-1}\), \(s=-t\), and
\(w=e^{2\pi is}\), take the unnormalized periods
\[
 P_{\infty,\lambda}(s)
 =
 P_{\mathrm{tw}}(-\lambda w)
 \begin{pmatrix}I_2&sB\\0&I_2\end{pmatrix},
 \qquad 0\le\lambda\le1.
\]
They satisfy
\[
 P_{\infty,\lambda}(s+1)
 =P_{\infty,\lambda}(s)M,
\]
and for \(|w|\in(0,1]\),
\[
 \det L_{P_{\infty,\lambda}(s)}
 =
 (1+|\lambda w|^2)^2
 -\frac{\log |w|}{2\pi}(1-|\lambda w|^4)
 >0.
\]
Near \(w=0\), the normalization is
\[
 P_{\infty,\lambda}(s)
 =
 A(-\lambda w)
 (I_2\;\;sB+C(-\lambda w)).
\]
Thus the determinant calculation and uniform filling estimates
are the same as at \(0\). The proof of
\cref{prop:parametric} gives completed blocks
\(\mathcal B_{\infty,\lambda}\to K_\infty\) and a flow
\[
 H_\infty:
 \mathcal B_{\infty,0}
 \xrightarrow{\;\cong\;}
 \mathcal B_{\infty,1},
\]
preserving the base and the full real period coordinates
on an annulus about the cutting circle.

To obtain \(D_\infty\), we identify
\(\mathcal B_{\infty,1}\) with the original completed end
and \(\mathcal B_{\infty,0}\) with the product model.

\smallskip
\noindent\emph{(a) The case \(\lambda=1\).}
On the regular part, the period matrices satisfy
\[
 P_{\infty,1}(s)
 =
 \begin{pmatrix}0&w\\-w&0\end{pmatrix}P(-s)G_0.
\]
Since \(G_0\Z^4=\Z^4\), this identity induces the fibrewise
isomorphism
\[
 \begin{aligned}
 \C^2/P(-s)\Z^4
 &\xrightarrow{\;\cong\;}
 \C^2/P_{\infty,1}(s)\Z^4,\\
 [(\zeta_1,\zeta_2)]
 &\longmapsto[(w\zeta_2,-w\zeta_1)].
 \end{aligned}
\]

It remains to check compatibility with the fillings.
Near \(w=0\), both constructions use the same normalized
filling model from \cref{rem:infinity-filling}, with
periods
\[
 (I_2\;\;sB+C(-w))=(I_2\;\;Z_{\infty}(s)).
\]
On the punctured overlap \(0<|w|<\epsilon\), the filling
comparisons form the diagram
\[
 \begin{tikzcd}
 \C^2/(I_2\;\;Z_\infty(s))\Z^4
   \arrow[r,"A(w^{-1})"]
   \arrow[d,equal]
 &
 \C^2/P(-s)\Z^4
   \arrow[d,"\cong"]
 \\
 \C^2/(I_2\;\;Z_\infty(s))\Z^4
   \arrow[r,"A(-w)"']
 &
 \C^2/P_{\infty,1}(s)\Z^4.
 \end{tikzcd}
\]
The diagram commutes because
\[
\begin{pmatrix}0&w\\-w&0\end{pmatrix}
A(w^{-1})=A(-w).
\]
Thus the regular-part isomorphism and the identity on the
common filling model fit together to define a base-preserving
diffeomorphism
\[
 \overline{\mathcal G}_\infty|_{K_\infty}
 \xrightarrow{\;\cong\;}
 \mathcal B_{\infty,1}.
\]

\smallskip
\noindent\emph{(b) The case \(\lambda=0\).}
The periods reduce to
\[
 P_{\infty,0}(s)
 =
 \begin{pmatrix}
  1&0&i&0\\
  0&1&0&s+i
 \end{pmatrix}
 =P_0(s).
\]
The product identification in \cref{prop:relative-tate} therefore
gives
\[
 J_\infty:
 \mathcal B_{\infty,0}
 \xrightarrow{\;\cong\;}E_i\times N_\infty.
\]
In the following, \(\mathcal B_{\infty,1}\) is identified with
\(\overline{\mathcal G}_\infty|_{K_\infty}\) by the
diffeomorphism in (a). Composing the two diffeomorphisims, we get 
\[
 \overline{\mathcal G}_\infty|_{K_\infty}
 \xrightarrow{\ H_\infty^{-1}\ }
 \mathcal B_{\infty,0}
 \xrightarrow{\ J_\infty\ }
 E_i\times N_\infty.
\]
Define
\[
 D_\infty=J_\infty\circ H_\infty^{-1}.
\]

On the regular part, the identification in (a) changes the real
period marking from \(v\) to \(G_0v\). The annular formula therefore follows
from the changes of base
and marking:
\[
 \begin{gathered}
 q=re^{2\pi i\theta},
 \quad
 w=r^{-1}e^{-2\pi i\theta},
 \quad
 s=-t=-\theta+\frac{\log(r^{-1})}{2\pi i},\\
 v\longmapsto G_0v
 =(v_1,v_2,-v_3,-v_4)^{\mathsf T}.
 \end{gathered}
\]
The flow preserves this transformed real period
vector. At the product endpoint, we therefore have
\[
 P_0(s)G_0v
 =
 \begin{pmatrix}
  1&0&i&0\\
  0&1&0&s+i
 \end{pmatrix}
 \begin{pmatrix}
  v_1\\v_2\\-v_3\\-v_4
 \end{pmatrix}
 =
 \begin{pmatrix}
  v_1-iv_3\\
  v_2+(s+i)(-v_4)
 \end{pmatrix}.
\]
The first component gives the \(E_i\)-coordinate, while
the second one corresponds to the Tate coordinates
\[
 [v_2+(s+i)(-v_4)]_{\Z+(s+i)\Z}
 \longleftrightarrow
 (r^{-1},[-\theta,v_2,-v_4]).
\]
Thus the annular formula at \(\infty\) is
\begin{equation}\label{eq:infinity-D-boundary}
 D_\infty(q,[P(t)v])
 =
 \bigl(
 [v_1-iv_3]_{\Z+i\Z},
 (r^{-1},[-\theta,v_2,-v_4])
 \bigr).
\end{equation}
As at \(0\), this formula holds on an open annulus and describes
the extension of \(D_\infty\) across the cutting circle.
\end{remark}

The maps \(D_0,D_\infty\) are smooth product identifications;
they do not assert holomorphic product decompositions of the
original completed families. In particular,
\eqref{eq:whole-block-collar} and
\eqref{eq:infinity-D-boundary} prepare the coordinates needed to
compute the smooth boundary map in \cref{sec:recognition}.

\section{Holomorphic gluing}\label{sec:global}

We will glue the two completions constructed in \cref{sec:filling}
along their common regular part \(\mathcal G\), using a fibrewise
holomorphic translation. The result is a compact complex threefold
\(X\) with a holomorphic map \(g:X\to\Pone\).
After constructing the translating section and computing its real
period coordinates on the cutting circle, we form the quotient
and determine the fibres and critical locus of \(g\).
The boundary calculation in \cref{sec:recognition} will show how
the chosen translation ensures that the resulting four-dimensional
factor is simply connected.

\subsection{The translating section}\label{sec:affine}

The linear change of marking alone does not specify the gluing:
we must also choose a translation in each torus fibre. The
holomorphic section constructed below supplies this translation
and has explicit boundary values. For \(q=e^{2\pi it}\)
and \(p_j(t)=P(t)e_j\), set
\begin{equation}\label{eq:u}
 u(t):=
 -\frac{t(t+1)}2p_2(t)+t\,p_4(t)=
 \binom{
 q\bigl(-\frac{t(t-1)}2+it\bigr)
 }{
 \frac{t(t-1)}2+it
 }.
\end{equation}
A direct calculation shows 
\begin{equation*}
 u(t+1)-u(t)
 =p_4(t)\in P(t)\Z^4.
\end{equation*}
Thus
\[
 [u(t+1)]_{P(t+1)\Z^4}
 =
 [u(t)]_{P(t)\Z^4},
\]
and \(\sigma(q):=[u(t)]\in\mathcal G_q\) defines a single-valued
holomorphic section.
Its fibrewise translation is
\[
 \tau_\sigma(q,[\zeta])=(q,[\zeta+u(t)])
\]
with \(\tau_\sigma^{-1}=\tau_{-\sigma},\ 
 p\circ\tau_\sigma=p\).

\begin{proposition}
\label[proposition]{prop:primitive}
On the boundary circle \(S=\{|q|=1\}\), we write
\(q=e^{2\pi i\theta}, \ \theta\in\R\).
The restriction \(\sigma|_S\) has real period lift
\begin{equation*}
 \mathbf a(\theta):=
 \left(
 0,-\frac{\theta(\theta+1)}2,0,\theta
 \right)^{\mathsf T}.
\end{equation*}
 In real period coordinates, this translation is
\[
 (\theta,[v])
 \longmapsto
 (\theta,[v+\mathbf a(\theta)]).
\]
\end{proposition}

\begin{proof}
On \(S\), the logarithmic parameter \(t=\theta\) is real.
Equation~\eqref{eq:u} gives
\(u(\theta)=P(\theta)\mathbf a(\theta)\).
Since \(\mathbf a(\theta)\in\R^4\), this is a real period lift
of \(\sigma|_S\). Thus the translation acts by
\[
 (q,[P(\theta)v])
 \longmapsto
 (q,[P(\theta)(v+\mathbf a(\theta))]),
\]
which verifies the stated formula in real period coordinates.

\end{proof}

\subsection{Gluing the completed ends}
We use \(\tau_\sigma\) to identify the regular parts of the two
completed ends. Recall from \cref{cor:both-fillable} that
\[
 \overline p_j:\overline{\mathcal G}_j\to U_j,
 \qquad
 \jmath_j:\mathcal G
 \xrightarrow{\;\cong\;}
 \overline{\mathcal G}_j|_{\C_q^*},
 \qquad j=0,\infty.
\]
Their bases satisfy
\[
 U_0\cup U_\infty=\Pone,
 \qquad
 U_0\cap U_\infty=\C_q^*,
 \qquad
 w=q^{-1}.
\]

Through the identifications \(\jmath_0\) and
\(\jmath_\infty\), the translation \(\tau_\sigma\) constructs a holomorphic isomorphism \(\Phi\) between
the regular parts of the two completions.
The construction and its compatibility with the projections
are shown in the following commutative cube:
\[
\begin{tikzcd}[column sep=2.8em,row sep=2.8em]
&
\overline{\mathcal G}_0|_{\C_q^*}
  \arrow[rr,"\Phi"]
  \arrow[dl,"\jmath_0^{-1}" description]
  \arrow[ddd,blue!65!black,
    "\overline p_0"{pos=0.63}]
&&
\overline{\mathcal G}_\infty|_{\C_q^*}
  \arrow[dl,"\jmath_\infty^{-1}" description]
  \arrow[ddd,blue!65!black,
    "\overline p_\infty"{pos=0.55}]
\\
\mathcal G
  \arrow[ddd,blue!65!black,"p"']
&&
\mathcal G
&
\\
&&&
\\
&
\C_q^*
  \arrow[rr,"q\mapsto q^{-1}"{pos=0.75}]
  \arrow[dl,equal]
&&
\C_w^*
  \arrow[dl,"w\mapsto w^{-1}" description]
\\
\C_q^*
  \arrow[rr,equal]
&&
\C_q^*
  \arrow[from=2-3,to=5-3,
    crossing over,blue!65!black,"p"{pos=0.55}]
  \arrow[from=2-1,to=2-3,
    crossing over,"\tau_\sigma"'{pos=0.68}]
&
\end{tikzcd}
\]
The upper face gives
\begin{equation}\label{eq:two-end-gluing-map}
 \Phi
 =
 \jmath_\infty\circ\tau_\sigma\circ\jmath_0^{-1},
\end{equation}
while the side faces express compatibility with the projections.
On the lower face, \(q\) and \(w=q^{-1}\) represent the same
point of \(\Pone\). Thus \(\Phi\) preserves the base point.

\begin{proposition}\label[proposition]{prop:descent}
The quotient
\begin{equation*}
 X=
 \overline{\mathcal G}_0
 \cup_\Phi
 \overline{\mathcal G}_\infty
\end{equation*}
is a compact complex threefold.
The projections induce a proper holomorphic map
\(g:X\to\Pone\) whose critical locus is the disjoint union of two
complex tori of complex dimension \(1\).
The fibres over \(0\) and \(\infty\) satisfy
\[
 g^{-1}(0)\cong g^{-1}(\infty)
 \cong E_i\times\bigl(\Pone/(0\sim\infty)\bigr).
\]
\end{proposition}

\begin{proof}
Since the projections satisfy
\(\overline p_\infty\circ\Phi=\overline p_0\) on the overlap,
they descend to a map
\(g:X\to\Pone\) with
\[
 g^{-1}(U_j)\cong\overline{\mathcal G}_j,
 \qquad
 g|_{g^{-1}(U_j)}=\overline p_j,
 \qquad j=0,\infty.
\]

Points with different images under \(g\) have disjoint
neighbourhoods pulled back from \(\Pone\). Points in the same
fibre lie in one of the open pieces
\(g^{-1}(U_j)\cong\overline{\mathcal G}_j\), where they can be
separated. Hence \(X\) is Hausdorff. The two countable
holomorphic atlases agree under \(\Phi\), so \(X\) is a complex
threefold and \(g\) is holomorphic.

The compact sets \(K_0,K_\infty\) cover \(\Pone\), and
properness of \(\overline p_0,\overline p_\infty\) shows
\[
 \begin{gathered}
 g^{-1}(K_j)\cong\overline p_j^{-1}(K_j)
 \quad\text{is compact},\qquad j=0,\infty,\\
 X=g^{-1}(K_0)\cup g^{-1}(K_\infty).
 \end{gathered}
\]
Hence \(X\) is compact, and \(g\) is proper.

The gluing takes place over \(\C^*\), so it leaves a neighbourhood
of each added fibre unchanged. By
\cref{prop:rank-one-fill,rem:infinity-filling},
\[
 \begin{aligned}
 g^{-1}(0)
 &\cong\overline p_0^{-1}(0)
 \cong E_i\times\bigl(\Pone/(0\sim\infty)\bigr),\\
 g^{-1}(\infty)
 &\cong\overline p_\infty^{-1}(\infty)
 \cong E_i\times\bigl(\Pone/(0\sim\infty)\bigr).
 \end{aligned}
\]
Over \(\C_q^*\), the map \(g\) is identified with the
submersion \(p:\mathcal G\to\C_q^*\).
Thus its only critical points are those of the two
fillings:
\[
 \operatorname{Crit}(g)
 =
 \operatorname{Crit}(\overline p_0)
 \sqcup
 \operatorname{Crit}(\overline p_\infty)
 \cong E_i\sqcup E_i.
\]

\end{proof}

\section{The diffeomorphism type of the total space}
\label{sec:recognition}

We complete the proof of \cref{thm:main} by determining the
diffeomorphism type of \(X\). The product identifications from
\cref{sec:relative}, together with the translation from
\cref{sec:global}, give an explicit boundary map.
This map yields a decomposition
\(X\cong_{\mathrm{diff}}\T^2\times Y\),
where \(Y\) is a closed smooth four-manifold.
Its fundamental group and Euler characteristic are computed
from the attaching map and \cref{lem:node-topology}, giving
\(\pi_1(Y)=0\) and \(\chi(Y)=2\).
As is well known, every simply connected closed smooth four-manifold
with Euler characteristic \(2\) is homotopy equivalent to \(S^4\).
It follows that \(Y\) is a smooth homotopy four-sphere.

It remains to identify the product \(\T^2\times Y\).
Hirsch's product theorem \cite[Theorem~6(a)]{Hirsch1965}
states that, for a smooth homotopy four-sphere \(Y\) and
a closed smooth manifold \(V\),
\[
V\times Y\cong_{\mathrm{diff}}V\times S^4.
\]
Applying this theorem with \(V=\T^2\), we obtain
\[
X\cong_{\mathrm{diff}}\T^2\times Y
\cong_{\mathrm{diff}}\T^2\times S^4.
\]
This argument does not require \(Y\) itself to be
diffeomorphic to the standard smooth \(S^4\).

\subsection{A product decomposition}
\label{subsec:product-over-base}

We express the boundary restriction of the holomorphic gluing map
in the smooth product coordinates. It is the product of a fixed
reflection of \(\T^2\) and a diffeomorphism between the boundaries
of the Tate blocks. Since the reflection extends over one product
block, it can be removed by changing the product identification
on that block.

For the discs \(K_0,K_\infty\) with common boundary \(S\) chosen
in \cref{sec:relative}, write \(\varpi_j:N_j\to K_j\) for the
Tate block projections. The real Tate coordinates introduced in
\cref{sec:topology} identify each boundary with the mapping torus
\[
 \partial N_j
 =\frac{\R_\theta\times\T^2_{\mu,\nu}}
 {(\theta+1,\mu,\nu)\sim(\theta,\mu+\nu,\nu)},
 \qquad j=0,\infty.
\]
Here the identification is exactly
\eqref{eq:tate-coordinate-seam}, and the coordinates
\([\theta,\mu,\nu]\) carry the markings fixed there.
We also use the identification
\[
 E_i\longrightarrow\T^2,
 \quad [x+iy]_{\Z+i\Z}\longmapsto[x,y].
\]
In these coordinates, the diffeomorphisms in
\cref{prop:relative-tate,rem:infinity-product} become
\[
 D_j:\overline{\mathcal G}_j|_{K_j}
 \xrightarrow{\;\cong_{\mathrm{diff}}\;}\T^2\times N_j.
\]

\begin{proposition}\label[proposition]{prop:product-splitting}
The map
\begin{equation}\label{eq:actual-four-dimensional-gluing}
 \begin{aligned}
 \psi:\partial N_0&\longrightarrow\partial N_\infty,\\
 [\theta,\mu,\nu]&\longmapsto
 \left[-\theta,\mu-\frac{\theta(\theta+1)}2,-\nu-\theta\right]
 \end{aligned}
\end{equation}
is a diffeomorphism. With the standard smooth structure
obtained by gluing along \(\psi\), the space
\[
 Y=N_0\cup_\psi N_\infty
\]
is a closed smooth four-manifold, and there is a
diffeomorphism
\begin{equation*}
 F:X\xrightarrow{\;\cong_{\mathrm{diff}}\;}\T^2\times Y.
\end{equation*}
\end{proposition}

\begin{proof}
\noindent\emph{Step 1. The boundary map.}
We write \(\Phi_S\) for the restriction of \(\Phi\) over \(S\).
In the product coordinates, the boundary transition is
\[
\begin{tikzcd}[column sep=large,row sep=large]
 \overline{\mathcal G}_0|_S
   \arrow[r,"\Phi_S"]\arrow[d,"D_0"']
 &\overline{\mathcal G}_\infty|_S\arrow[d,"D_\infty"]\\
 \T^2\times\partial N_0\arrow[r,"f_S"']
 &\T^2\times\partial N_\infty,
\end{tikzcd}
\qquad
 f_S:=D_\infty\circ\Phi_S\circ D_0^{-1}.
\]
For \(q=e^{2\pi i\theta}\) and
\(v=(x,\mu,y,\nu)^{\mathsf T}\),
equations~\eqref{eq:whole-block-collar} and
\eqref{eq:infinity-D-boundary} imply
\[
 \begin{aligned}
 (D_0\circ\jmath_0)(q,[P(\theta)v])
   &=([x,y],[\theta,\mu,\nu]),\\
 (D_\infty\circ\jmath_\infty)(q,[P(\theta)v])
   &=([x,-y],[-\theta,\mu,-\nu]).
 \end{aligned}
\]
The radial coordinate is \(1\) on both boundaries and is omitted.
By \eqref{eq:two-end-gluing-map},
\[
 \Phi_S\bigl(\jmath_0(q,[P(\theta)v])\bigr)
 =\jmath_\infty(q,[P(\theta)(v+\mathbf a(\theta))]),
\]
where
\[
 \mathbf a(\theta)=
 \left(0,-\frac{\theta(\theta+1)}2,0,\theta\right)^{\mathsf T}
\]
is the lift from \cref{prop:primitive}.
Applying \(D_\infty\) gives the explicit formula
\begin{equation*}
 \begin{aligned}
 f_S([x,y],[\theta,\mu,\nu])
 &=\left([x,-y],
   \left[-\theta,\mu-\frac{\theta(\theta+1)}2,-\nu-\theta\right]\right)\\
 &=(R_T\times\psi)([x,y],[\theta,\mu,\nu]),
 \end{aligned}
\end{equation*}
where \(R_T[x,y]=[x,-y]\), as in~\eqref{eq:G0}, and
\(
\psi[\theta,\mu,\nu] = \left[-\theta,\mu-\frac{\theta(\theta+1)}2,-\nu-\theta\right]
\).

The formula for \(\psi\) is compatible with the boundary
identifications:
\[
 \begin{aligned}
 \psi[\theta+1,\mu,\nu]
 &=\left[-\theta-1,
     \mu-\frac{(\theta+1)(\theta+2)}2,-\nu-\theta-1\right]\\
 &=\left[-\theta,
     \mu+\nu-\frac{\theta(\theta+1)}2,-\nu-\theta\right]\\
 &=\psi[\theta,\mu+\nu,\nu].
 \end{aligned}
\]
The inverse is
\[
 \psi^{-1}[\eta,\alpha,\beta]
 =\left[-\eta,\alpha+\frac{\eta(\eta-1)}2,-\beta+\eta\right].
\]
Thus \(\psi:\partial N_0\to\partial N_\infty\) is a
diffeomorphism.

\smallskip
\noindent\emph{Step 2. The smooth structure on \(Y\).}
Using the standard collar-gluing construction
\cite[Chapter~8, Section~2]{Hirsch1976}, we obtain a closed smooth
four-manifold \(Y=N_0\cup_\psi N_\infty\).
Its diffeomorphism type is independent of the collar choices
by the uniqueness of gluing
\cite[Chapter~8, Theorem~2.1]{Hirsch1976}.
Since \(D_0,D_\infty\) extend across the cutting circle, the
same collar-gluing theorem applies to the decomposition of
\(X\) along \(g^{-1}(S)\). The boundary formula
\(f_S=R_T\times\psi\) gives
\[
 X\cong_{\mathrm{diff}}
 (\T^2\times N_0)
 \cup_{R_T\times\psi}
 (\T^2\times N_\infty).
\]
We have
\[
 (\Id_{\T^2}\times\psi)\circ(R_T\times\psi)^{-1}
 =R_T\times\Id_{\partial N_\infty}.
\]
This boundary diffeomorphism
\(R_T\times\Id_{\partial N_\infty}\) extends over the
second block as
\[
 R_T\times\Id_{N_\infty}:
 \T^2\times N_\infty
 \longrightarrow\T^2\times N_\infty.
\]
The condition in
\cite[Chapter~8, Theorem~2.2]{Hirsch1976}
is therefore satisfied, and we obtain
\[
 \begin{aligned}
 X
 &\cong_{\mathrm{diff}}
 (\T^2\times N_0)
 \cup_{R_T\times\psi}(\T^2\times N_\infty)\\
 &\cong_{\mathrm{diff}}
 (\T^2\times N_0)
 \cup_{\Id_{\T^2}\times\psi}(\T^2\times N_\infty)\\
 &\cong_{\mathrm{diff}}\T^2\times Y.
 \end{aligned}
\]
The diffeomorphism can be chosen relative to the first block.
Indeed, the uniqueness-of-gluing theorem fixes that block
pointwise \cite[Chapter~8, Theorem~2.1]{Hirsch1976},
and the reflection above acts only on the second block.
Thus, writing \(\iota_0:N_0\hookrightarrow Y\) for the natural
inclusion and identifying \(g^{-1}(K_0)\) with
\(\overline{\mathcal G}_0|_{K_0}\), we may choose \(F\) so that
\begin{equation}\label{eq:relative-product-first-block}
 F|_{g^{-1}(K_0)}
 =(\Id_{\T^2}\times\iota_0)\circ D_0.
\end{equation}

\end{proof}

\subsection{The four-dimensional factor is a homotopy sphere}
\label{subsec:homotopy-factor}

We next compute the effect of \(\psi\) on the three marked
boundary loops. The image of the base loop
\(\theta_0\) is the essential point: although \(\theta_0\)
is trivial in \(N_0\), its image contains the surviving
longitude of \(N_\infty\). This gives the additional relation
that kills the fundamental group of \(Y\).

\begin{proposition}\label[proposition]{prop:homotopy-sphere}
The manifold \(Y\) is a smooth homotopy four-sphere.
\end{proposition}

\begin{proof}
Recall from \cref{sec:topology} that the marked
loops in \(\partial N_j\), based at \([0,0,0]\), are
\[
 \begin{aligned}
 m_j(s)&=[0,s,0],\\
 \ell_j(s)&=[0,0,s],\\
 \theta_j(s)&=[s,0,0],
 \end{aligned}
 \qquad 0\le s\le1.
\]
By \cref{lem:node-topology},
\[
 \pi_1(N_j)=\Z\langle\ell_j\rangle,
 \qquad
 m_j=\theta_j=1
 \quad\text{in }\pi_1(N_j).
\]
Formula~\eqref{eq:actual-four-dimensional-gluing} gives
\[
 \psi\circ m_0=m_\infty,
 \quad
 \psi\circ\ell_0=\ell_\infty^{-1}.
\]
To compute the image of \(\theta_0\), consider
\[
 [0,1]^2\longrightarrow\partial N_\infty,
 \quad
 (s,t)\longmapsto
 \left[-s,-\frac{(s+1)t}{2},-t\right].
\]
Its diagonal is
\[
 s\longmapsto
 \left[-s,-\frac{s(s+1)}2,-s\right]
 =\psi\bigl(\theta_0(s)\bigr).
\]
The bottom and right edges are, respectively,
\[
 s\longmapsto[-s,0,0]=\theta_\infty^{-1}(s),
 \qquad
 t\longmapsto[-1,-t,-t]
 =[0,0,-t]=\ell_\infty^{-1}(t),
\]
where the last identification follows from the
mapping torus relation~\eqref{eq:tate-coordinate-seam}.
The diagonal is homotopic to the bottom edge followed by the right edge, relative to its endpoints. Hence
\[
 \psi\circ\theta_0
 \simeq\theta_\infty^{-1}*\ell_\infty^{-1}
 \quad\text{in }\partial N_\infty.
\]
Since \(m_\infty\) and \(\theta_\infty\) are
null-homotopic in \(N_\infty\), we obtain
\begin{equation*}
 \psi\circ m_0\simeq1,\qquad
 \psi\circ\ell_0\simeq\ell_\infty^{-1},\qquad
 \psi\circ\theta_0\simeq\ell_\infty^{-1}
 \quad\text{in }N_\infty.
\end{equation*}
The Seifert--van Kampen theorem therefore gives
\begin{equation*}
 \pi_1(Y)
 =
 \langle\ell_0,\ell_\infty\mid
 \ell_0=\ell_\infty^{-1},\
 1=\ell_\infty^{-1}\rangle
 =0.
\end{equation*}

By \cref{lem:node-topology}, each block has Euler
characteristic one. Since \(\partial N_0\) is a closed
three-manifold,
\begin{equation*}
 \chi(Y)
 =\chi(N_0)+\chi(N_\infty)-\chi(\partial N_0)
 =1+1-0=2.
\end{equation*}
Since \(Y\) is simply connected, it is orientable. Choose an
orientation. Poincar\'e duality and the universal coefficient
theorem give
\[
 H_2(Y;\Z)
 \cong H^2(Y;\Z)
 \cong\operatorname{Hom}(H_2(Y;\Z),\Z).
\]
Thus \(H_2(Y;\Z)\) is torsion-free. Moreover,
\[
 \chi(Y)=2+\operatorname{rank}H_2(Y;\Z)=2,
\]
so \(H_2(Y;\Z)=0\). Duality also gives \(H_3(Y;\Z)=0\).
Consequently,
\[
 H_*(Y;\Z)\cong H_*(S^4;\Z).
\]
Successive applications of the Hurewicz theorem yield
\[
 \pi_2(Y)=\pi_3(Y)=0,
 \quad
 \pi_4(Y)\cong H_4(Y;\Z)\cong\Z.
\]
A map \(S^4\to Y\) representing an oriented generator
of \(\pi_4(Y)\) is therefore an integral homology
equivalence. Since both spaces have the homotopy type
of simply connected CW complexes, this map is a
homotopy equivalence
\cite[Corollary~4.33]{Hatcher2002}.

\end{proof}

\begin{remark}\label[remark]{rem:fibre-generator}
Choose \(q_*\in\operatorname{Int}K_0\setminus\{0\}\) and
\(p\in\varpi_0^{-1}(q_*)\). Since \(D_0\) preserves the base,
the embedded real torus
\[
 T_*=D_0^{-1}(\T^2\times\{p\})
\]
lies in \(g^{-1}(q_*)\). By
\eqref{eq:relative-product-first-block},
\(F(T_*)=\T^2\times\{\iota_0(p)\}\).
Since \(Y\simeq S^4\), the K\"unneth formula shows that
\([T_*]\) generates \(H_2(X;\Z)\cong\Z\), after choosing an
orientation of \(T_*\). No holomorphicity of \(T_*\) is asserted.
\end{remark}







\begin{proof}[Proof of \cref{thm:main}]
\Cref{prop:descent} gives a compact complex threefold \(X\) and a
surjective holomorphic map \(g:X\to\Pone\). By
\cref{prop:product-splitting,prop:homotopy-sphere} and Hirsch's
theorem \cite[Theorem~6(a)]{Hirsch1965}, the underlying smooth
manifold is diffeomorphic to \(\T^2\times S^4\). Pulling back
the complex structure by this diffeomorphism gives the required
complex structure.
Over \(\C^*\), the map \(g\) is the family of complex two-tori
constructed in \cref{prop:period-bridge}. By
\cref{prop:rank-one-fill,rem:infinity-filling,prop:descent}, its
critical locus consists of two disjoint complex tori of complex
dimension \(1\), lying over \(0\) and \(\infty\), respectively. This proves both
assertions of \cref{thm:main}.

\end{proof}

\section{The non-K\"ahler and non-product properties}
\label{sec:consequences}

With \cref{thm:main} proved, we establish the non-K\"ahler and
non-product assertions in \cref{prop:nonproduct}. These depend
only on the topology of \(\T^2\times S^4\) and hold for every
complex structure on that smooth manifold.

We use the cohomology ring to exclude a symplectic form, and
the characteristic numbers of a compact complex surface to
exclude a biholomorphic product decomposition.

\begin{proof}[Proof of \cref{prop:nonproduct}]
\noindent\emph{Non-K\"ahlerness.}
The K\"unneth formula identifies
\(H^2(Z;\R)\) with the pullback of \(H^2(\T^2;\R)\)
under a smooth product identification. Consequently every
\(\alpha\in H^2(Z;\R)\) satisfies \(\alpha^2=0\), and hence
\(\alpha^3=0\). If the underlying smooth manifold admitted a
symplectic form \(\omega\), its symplectic orientation would give
\(\int_Z\omega^3>0\), contradicting \([\omega]^3=0\).
Thus the underlying manifold admits no symplectic form, and in
particular \(Z\) admits no K\"ahler metric.

\smallskip
\noindent\emph{Non-productness.}
Suppose that \(Z\) admits a nontrivial biholomorphic product
decomposition. After grouping factors if necessary, we may write
\(Z\cong C\times S\), where \(C\) is a compact complex curve
of genus \(h\) and \(S\) is a compact complex surface.
The K\"unneth formula with real coefficients
\cite[Section~3.2]{Hatcher2002} gives
\[
\begin{aligned}
	2&=b_1(C\times S)=2h+b_1(S),\\
	1&=b_2(C\times S)=1+b_2(S)+2h\,b_1(S).
\end{aligned}
\]
Since the Betti numbers are nonnegative, these equalities imply
\(b_2(S)=0\) and \(b_1(S)\in\{0,2\}\).
Poincar\'e duality therefore gives
\[
\chi_{\mathrm{top}}(S)
=2-2b_1(S)+b_2(S)
\in\{-2,2\}.
\]

On the other hand, since \(b_2(S)=0\), we have \(H^2(S;\mathbb R)=0\). So the intersection form
has signature \(\sigma(S)=0\); see
\cite[Section~19]{MilnorStasheff1974}.
The first Chern class also vanishes in real cohomology, and hence
\(\langle c_1(TS)^2,[S]\rangle=0\).
The characteristic-class identities \(c_2=e\) and
\(p_1=c_1^2-2c_2\), together with the signature theorem
\cite[Section~14, Corollary~15.5, and Theorem~19.4]{MilnorStasheff1974}
(see also \cite{Hirzebruch1995}), now yield
\[
0=\langle c_1(TS)^2,[S]\rangle
=2\chi_{\mathrm{top}}(S)+3\sigma(S)
=2\chi_{\mathrm{top}}(S).
\]
Thus \(\chi_{\mathrm{top}}(S)=0\), contradicting
\(\chi_{\mathrm{top}}(S)\in\{-2,2\}\).
Therefore \(Z\) admits no nontrivial biholomorphic product
decomposition.

\end{proof}







\noindent
\textbf{AI Disclosure}.
The authors used Codex (OpenAI) to assist with literature searches
and the development and checking of mathematical constructions
and arguments. The authors take full responsibility for the statements,
proofs, and references in this paper.

\end{document}